\documentclass[10pt,twoside,a4paper]{article}
\usepackage{amsmath}
\usepackage{amsthm}
\usepackage{amssymb}
\usepackage{hyperref}
\usepackage{amssymb}
\usepackage{graphicx}
\usepackage{tabularx}
\usepackage{multirow}
\usepackage{fullpage}
\usepackage{color}
\usepackage{xspace}
\usepackage{tikz,pgfplots} 
\usepackage[vlined,longend,linesnumbered,tworuled]{algorithm2e}
\newtheorem{mydef}{Definition}[section]
\newtheorem{theo}{Theorem}[section]
\newtheorem{prop}{Proposition}[section]
\newtheorem{remark}{Remark}[section]

\newtheorem{assumption}{Assumption}[section]

\newtheorem{mylemma}{Lemma}[section]

\usepackage[software,hardware]{mymacros}

\newcommand{\Pinf}{P_{\infty}}
\newcommand{\Qinf}{Q_{\infty}}
\newcommand{\ds}{\displaystyle}
\newcommand{\argmax}{\ensuremath{\mathop{\mathrm{argmax}}}}
\newcommand{\smb}{\left[\begin{smallmatrix}}
\newcommand{\sme}{\end{smallmatrix}\right]}

	\newlength\wex
	\newlength\hex
\definecolor{mycolor1}{rgb}{0.00000,0.44700,0.74100}%
\definecolor{mycolor2}{rgb}{0.85000,0.32500,0.09800}%
\definecolor{mycolor3}{rgb}{0.92900,0.69400,0.12500}%
\definecolor{mycolor4}{rgb}{0.49400,0.18400,0.55600}%
\definecolor{mycolor5}{rgb}{0.46600,0.67400,0.18800}%

 \pgfplotscreateplotcyclelist{error}{%
 {color=blue, solid, mark=+,line width=1pt, mark options={solid,mark size=3pt},mark repeat=50},
 {color=red,solid,mark=*,mark size=1.5pt,line width=1pt, mark repeat=50},
 {color=blue,densely dashed,mark=triangle*,mark size=1.5pt,line width=1pt, mark repeat=50},
 {color=blue,densely dashed,mark=square*,mark options={solid,mark size=1.5pt},line width=1pt, mark repeat=50},
 {color=red, dashdotted, mark=triangle*, mark options={solid,mark size=1.5pt},line width=1pt, mark repeat=50},
 {color=red, dashdotted, mark=diamond, mark options={solid,mark size=1.5pt},line width=1pt, mark repeat=50},
 }

\title{Numerical computation and new output bounds for time-limited balanced truncation of discrete-time systems}

\author{Igor Pontes Duff\footnotemark[1] \and Patrick K{\"u}rschner\footnotemark[1]~\footnotemark[2]} 
\begin{document}
\footnotetext[1]{Max Planck Institute for Dynamics of Complex Technical Systems, Magdeburg, Germany}
 \footnotetext[2]{KU Leuven, Electrical Engineering (ESAT), Kulak Kortrijk Campus, Belgium}
\maketitle

\begin{abstract}
	In this paper, balancing based model order reduction (MOR) for large-scale linear discrete-time time-invariant systems in prescribed finite time intervals is studied. The first main topic is the development of error bounds regarding the approximated output vector within the time limits. The influence of different components in the established bounds will be highlighted. After that, the second part of the article proposes strategies that enable an efficient numerical execution of time-limited balanced truncation for large-scale systems. Numerical experiments illustrate the performance of the proposed techniques.
\end{abstract}

%\tableofcontents

\section{Introduction}

In this paper, we consider multi-input multi-output (MIMO) linear time-invariant (LTI) discrete-time dynamical systems.  These systems are governed by a set a difference equations of the form

\begin{equation}\label{eq:LTI}
	\cS : \left\{\begin{array}{rl}
		x(k+1) &= Ax(k) +Bu(k),~\textrm{for}~\,k\in \N =\{0, 1, 2, \dots\}\\
		y(k) &= Cx(k) ,~\,~x(0) = x_0,
	\end{array}\right.
\end{equation}
where $x(k)\in \R^{n}$ is the state-variable,  $u(k)\in \R^m$ is the input, $y(k)\in \R^p$ is the output for every discrete-time $k \in \N $. Here $A\in \R^{n \times n}$, $B \in \R^{n\times m}$,  $C\in \R^{p \times n}$ and the leading dimension $n$ is the order of the system. We denote $\cS = (A, B ,C)$ for the given realization \eqref{eq:LTI}.  We assume that $x_0 = 0$ and reader is referred to \cite{morHeiRA11,baur2014model,beattie2017model} which treat the case of nonzero initial condition for continuous-time systems. 

In this case, we can represent the output as 
\begin{equation}
	y(k) = \sum_{j=0}^{k} h(k-j)u(j) = (h\ast u)(k), 
\end{equation}
where $h$ is the impulse response of the system, given by
\begin{equation}\label{eq:ImpResp}
	h(0) = 0,~h(k) = CA^{k-1}B,~\,\textmd{for}~\,~k =1, 2, \dots.   
\end{equation}
We say that $\cS$ is (asymptotically) stable, if and only if $A$ has its eigenvalues inside the unitary disc, in which case we call the matrix $A$ stable. Otherwise, we say that $A$ is unstable.  For stable systems,  the infinite reachability and observability Gramians $P_{\infty}$ and $Q_{\infty}$ are defined as 
\begin{subequations}\label{eq:InfGramians}
	\begin{align}
		P_{\infty} &= \sum_{k=1}^{\infty} A^{k-1}B\left(A^{k-1}B\right)^T, \\
		Q_{\infty} &= \sum_{k=1}^{\infty} \left(CA^{k-1}\right)^TCA^{k-1},
	\end{align}
\end{subequations}
and they are the unique solution of the following Stein  equations (discrete-time Lyapunov equations)
\begin{subequations}\label{eq:InfSteinEq}
	\begin{align}
		A\Pinf A^T - \Pinf +BB^T &= 0, \\
		A^T\Qinf A^T - \Qinf + C^TC &= 0.
	\end{align}
\end{subequations}
A LTI discrete-time system~\eqref{eq:LTI} is said to be minimal in infinite horizon if $\Pinf\Qinf$ is nonsingular.

Mathematical models of systems~\eqref{eq:LTI} are considered to be large scale, whenever its order is very large, perhaps  $n>10^5$ or more.  This leads to difficulties for tasks involving simulation, optimization or control of this system, motivating the use of a reduced order model (ROM) of the form 
\begin{equation}\label{eq:ROM}
	\hat{\cS} : \left\{\begin{array}{rl}
		\hx(k+1) &= \hat A\hx(k) +\hat B u(k),~\textrm{for}~\,k\in \N \\\
		\hy(k) &= \hat C\hx(k), 
	\end{array}\right.
\end{equation}
where $\hx(k)\in \R^r$, for $k\in \N$, $\hA\in \R^{r \times r}$, $\hB \in \R^{r \times m}$ and $\hC \in \R^{m \times r}$.  The goal is to construct an $\hat{\cS}$ such that $r \ll n$ and still $\hat{y} \approx y$ for a large class of inputs $u$.  Projection based model reduction consists in constructing matrices $W, V \in \R^{n \times r}$ with $W^TV = I_{r}$, such that 
\begin{equation}\label{eq:ProjMatrices}
	\hA = W^TAV,~\,~\hB = W^TB~\,~\textnormal{and}~\,~\hC = CV.
\end{equation}

In order to measure the quality of reduced order models, system norms are defined. Given a stable system $\cS$ as in \eqref{eq:LTI} whose impulse response $h$ is given by \eqref{eq:ImpResp}, its $h_{\infty}$ and  $h_2$ norms are defined as
\begin{align}
\begin{split}
\|\cS\|_{h_{\infty}} &= \sup_{w\in [0,2\pi]} \|C(e^{i\omega}I-A)^{-1}B\|_2,~\hfill~\text{and}~\hfill  
\\ \|\cS\|_{h_{2}}  &=  \left(\sum_{j=0}^{\infty}\|h(j)\|_{F}^2\right)^{1/2}  = \trace{CP_{\infty}C^T)}^{1/2} = \trace{B^TQ_{\infty}B)}^{1/2}.
\end{split}
\end{align}

Balanced truncation (BT) is a model order reduction technique introduced in \cite{morMoo79} allowing to construct such a reduced order model $\hat{\cS}$ by projection. It relies on the concept of simultaneous diagonalization of the reachability and observability Gramians. In other words, the goal is to find a state-space transformation $T\in \R^{n \times n}$ nonsingular, such that \[T\Pinf T^T = T^{-T}\Qinf T^{-1} =  \Sigma_{\infty} = \begin{bmatrix}
\Sigma_{1,\infty} & 0 \\ 0 &\Sigma_{2,\infty}
\end{bmatrix},\]  where $\Sigma_{1, \infty} = \diag{\sigma_{1,\infty}, \dots, \sigma_{r,\infty}}$, $\Sigma_{2, \infty} = \diag{\sigma_{r+1,\infty}, \dots, \sigma_{n,\infty}}$, and $\sigma_{1,\infty} \geq \dots \geq \sigma_{n,\infty} \geq 0$ are the so-called Hankel singular values. Let \[TAT^{-1} := A_{\cB} = \begin{bmatrix}
A_{11} & A_{12} \\ A_{21} & A_{22}
\end{bmatrix}, TB := B_{\cB} = \begin{bmatrix}
B_1 \\ B_2
\end{bmatrix} , CT^{-1} := C_{\cB} = \begin{bmatrix}
C_1 & C_2
\end{bmatrix}. \] The equivalent realization $(A_{\cB}, B_{\cB}, C_{\cB})$ is referred to as the balanced realization.  Then, the projection matrices $V$ and $W$ are taking as the first $r$ columns of $T$ and $T^{-T}$, respectively, and the reduced order model is given by \eqref{eq:ProjMatrices}.   %It is important to emphasis that the balanced transformation $T$ does not need to be computed explicitly. Instead, let $\Pinf = S^TS$ and $\Qinf = RR^T$ be the Cholesky decomposition of the Gramians, and

The reduced order system $\hat{\cS}$ obtained by balancing satisfies an a priori error bound in the $h_\infty$ norm which is given by (cf. \cite[Theorem 7.10]{antoulas2005approximation})
\begin{equation}\label{eq:sumhsvBT}
	\|\cS -\hat{\cS}\|_{h_{\infty}}\leq 2\left(\sum_{k=r+1}^n \sigma_{k}\right) = 2\trace{\Sigma_{2, \infty}}=:\boldsymbol{\sigma}_{r},
\end{equation} 
\emph{i.e.}, the $h_{\infty}$ norm of the error system is bounded by twice the sum of the neglected Hankel singular values. This error bound is also valid in the continuous-time context due to \cite{morGlo84,morEnn84}.

An error bound a posteriori with respect to the $h_2$ norm is also available in \cite{chahlaoui2012posteriori}. It is expressed by\footnote{The expression appearing in \eqref{eq:InfHorErrorBound2}  differs from that presented in Theorem 2 in the paper\cite{chahlaoui2012posteriori}. In fact, we re-done the computation and realize that \eqref{eq:InfHorErrorBound2} is the correct expression.   }
\begin{subequations}\label{eq:InfHorErrorBound}
\begin{align}
\|\cS -\hat{\cS}\|_{h_2} &= \trace{C_2\Sigma_{2,\infty}C_2^T+ 2A_{12}\Sigma_{2,\infty}A^{T}_{:2}Z_{\infty}} + \trace{C_1(\hat P_{\infty} -\Sigma_{1,\infty})C_1^T} \\%+ 2\trace{A_{12}\Sigma_{2,\infty}A^{T}_{:2}Z_{\infty}} \\ 
%&= \trace{B_2^T\Sigma_2B_2} + \trace{B_1^T(\hat P_{\infty} -\Sigma_1)B_1^T} + 2\trace{A_{12}\Sigma_2A^{T}_{:2}Y_{\infty}}, 
&=\trace{B_2^T\Sigma_{2,\infty}B_2+ 2A_{21}^T\Sigma_{2,\infty}A_{2:}Y}  +\trace{B_1^T(\hQ_{\infty}-\Sigma_{1, \infty})B_1}  \label{eq:InfHorErrorBound2}
\end{align} 
\end{subequations}
where $\hat P_{\infty}$ and $\hat Q_{\infty}$ are, respectively, the reachability and observability Gramians of the ROM, which satisfy
\begin{align*}
A_{11}\hat P_{\infty}A_{11}^T -\hat P_{\infty} +B_1B_1^T &= 0, \\
A_{11}^T\hat Q_{\infty}A_{11} -\hat Q_{\infty} +C_1^TC_1 &= 0.
\end{align*} 
The matrices $Y_{\infty}, Z_{\infty}\in \R^{n\times r}$  are the solutions of the Sylvester equations
\begin{subequations}\label{eq:SylvStein_inf}
\begin{align}
AY_{\infty}A_{11}^T -Y_{\infty} +BB_1^T &= 0,\\
A^TZ_{\infty}A_{11} -Z_{\infty} +C^TC_1 &= 0, 
\end{align} 
\end{subequations}
and $A_{:2}^T = \begin{bmatrix}
A_{12}^T & A_{22}^T
\end{bmatrix}.$ It is worth noticing that an $H_2$ error bound for continuous-time systems is also available in \cite[Lemma~7.13]{antoulas2005approximation}. Similar research for stochastic systems can be found in, e.g.,
~\cite{morBenR15,morBenR17a,FreR18}. 

%and the following $h_2$ error bound (see \cite{chahlaoui2012posteriori})
%\begin{equation}
%\alpha \|C\|_2^2\sigma_{r+1}\geq \|\cS -\hat{\cS}\|_{h_{2}}^2\leq c p\|C\|_2^2\sigma_{r+1},
%\end{equation} 
%where
%\[c = 1+3\|A\|_{2}^2\ds\sum_{j=1}^{\infty}\|A^i\|_2\|\|(A_{ii})^i\|_2,~\alpha = \|C_{:1}\|_2^2/\|C\|_2^2, \]
%where $C_{:1}$ is the first column of $C$.

Balanced truncation for continuous- and discrete-time LTI systems was extended by the restriction to given time intervals  in \cite{gawronski1990model}. In this context, one aims at a ROM that is an accurate approximation until a finite time horizon $\tau>0$, but allows the ROM to be inaccurate outside of the time interval. The time-limited (TL) Gramians are defined as
\begin{subequations}\label{eq:TLGramians}
	\begin{align}
		P_{\tau} &= \sum_{k=1}^{\tau} A^{k-1}B\left(A^{k-1}B\right)^T, \\
		Q_{\tau} &= \sum_{k=1}^{\tau} \left(CA^{k-1}\right)^TCA^{k-1},
	\end{align}
\end{subequations}
and satisfy the following Stein equations
\begin{subequations}\label{eq:FiniteSteinEq}
	\begin{align}
		AP_{\tau}A^T - P_{\tau} +BB^T &= FF^T, \label{eq:FiniteSteinEq1} \\
		A^TQ_{\tau}A^T - Q_{\tau} + C^TC &= G^TG,\label{eq:FiniteSteinEq2}
	\end{align}
\end{subequations}
where $F = A^{\tau}B$ and $G = CA^{\tau}$. Even if the pairs $(A,B)$ and $(A^T,C^T)$ are reachable,  the TL Gramians~\eqref{eq:TLGramians} might be only positive semidefinite. This might happen whenever $\tau <n/m$ or $\tau <n/p$.  In this case, one can remove the states that are unreachable and unobservable for the the given time-interval, which are given by the kernels of $P_{\tau}$ and  $Q_{\tau}$. As a consequence, the resulting system is reachable and observable for the given time-interval and the Gramians in \eqref{eq:TLGramians} are positive definite matrices. Henceforth, we will assume that the TL Gramians in \eqref{eq:TLGramians} are positive definite matrices.

The time-limited balanced truncation (TLBT) is obtained by balancing $P_{\tau}$ and $Q_{\tau}$, \emph{i.e.}, finding the state transformation $T$ such that  $TP_{\tau} T^T = T^{-T}Q_{\tau} T^{-1} = \diag{\sigma_1, \dots, \sigma_r}$ and neglecting the states associated to small time-limited Hankel singular values.  Reader should notice that Gramians $P_{\tau}$ and $Q_{\tau}$ also exist in the case $A$ matrix is unstable provided $\forall\lambda\in\lambda(A)\backslash\lbrace 0\rbrace$ it holds $1/\lambda\notin\Lambda(A)$. As a consequence, TLBT is also applicable to unstable systems. On the other hand, for stable systems TLBT is not guaranteed to preserve the stability, but experimental evidence~\cite{morKue18,morRedK18} indicates that this does not deteriorate the approximation quality in the targeted time interval which will be also confirmed by the experiments in this paper. Also the upcoming error bounds will, to some extent, indicate that the occasionally generated unstable reduced order models still provide accurate output approximations. 
Some stability preserving variants of time-/ and frequency-limited BT have been proposed in, e.g., \cite{morGugA04, HaiGIM17,redmann2018output, ImrGUS18} leading to so called modified BT variants.
However, enforcing stability via such modified TLBT variants appears to deteriorate the good approximation quality of TLBT within the time interval and, at the same time, is computational more expensive~\cite{morBenKS16,Kue16,morKue18} for large systems. Hence, we will in the study at hand not consider such stability preserving variants. Additionally, readers should refer to \cite{goyal2017towards,sinani2018mathcal,VuilleminDARPOECC:2014,petersson2014model} for $\mathcal{H}_2$ time-/ and frequency-limited model reduction of continuous-time systems.

In this paper,  time-limited balanced truncation for large-scale linear discrete-time time-invariant systems is studied. The main contribution is twofold.  In the first part, we develop error bounds regarding the approximated output vector within the time limits. Those error bounds are an extension of those given in \cite{chahlaoui2012posteriori} to the time-limited case. However,  they also hold in the case the original system or the reduced order model are unstable.  Additionally,  their asymptotic behavior with respect to the time limits is analyzed and some sufficient conditions to preserve stability are provided.  The second part of the article proposes strategies that enable an efficient numerical execution of time-limited balanced truncation for large-scale systems which has so far not been considered in the literature. These strategies rely in solvers of the time-limited Stein equation using low-rank factors. Different solvers are proposed and their performance are compared. 

The rest of the paper is organized as follows. In Section \ref{sec:Prelim}, the time-limited $h_2$ inner product and norm are defined and characterized using Gramians. Also, a first error bound is provided based on the discrete-time convolution expression. In Section \ref{ref:ErrorBoundTLBT}, an tailored error bound for time-limited balanced truncation is developed. Additionally,  a sufficient condition for stability preservation is provided and the asymptotic behavior of the error bound is studied. In Section \ref{sec:compute}, different solvers based on low-rank factors are proposed to compute the TL Gramians.  Finally, Section \ref{sec:NumExp} carried out  some numerical experiments for  large-scale systems and Section \ref{sec:Conclusion} concludes the paper. 

\section{Preliminary results}\label{sec:Prelim}

\subsection{TL $h_2$ inner product and norm}
From now on, we consider the finite horizon $\tau$ to be fixed. In what follows, we recall the definition of the TL $h_2$ norm and inner-product. 

\begin{mydef}{\bf (time-limited  $h_2$ norm and inner-product)}\label{def:h2innernorm} Let $\cS = (A, B, C)$ and $\hat{\cS} = (\hA, \hB, \hC)$ be two LTI discrete-time dynamical systems as in \eqref{eq:LTI}. Then, the $h_2$ TL inner-product between $\cS$ and  $\hat{\cS}$ is given by 
	\begin{equation}\label{eq:TLinprod}
	\langle \cS, \hat{\cS} \rangle_{h_2,\tau} = \sum_{j=0}^{\tau} \trace{h(j)\hh(j)^T},
	\end{equation}
	where $h(0)=0, h(k) = CA^{k-1}B$ and $\hh(0) = 0, \hh(k)  = \hC\hA^{k-1}\hB$ for $k\in \N^{*}$ are, respectively, the impulse response of $\cS$ and $\hat{\cS}$. Moreover, the $h_2$ TL norm of $\cS$ is given by\footnote{Given a matrix $h\in \R^{p \times m}$, its Frobenius norm is defined as $\|h\|_F^2 = \trace{hh^T}$.} 
	\begin{equation}\label{eq:TLnorm}
	\|\cS\|_{h_2,\tau} = \left(\sum_{j=0}^{\tau} \trace{h(j)h(j)^T}\right)^{1/2} = \left(\sum_{j=0}^{\tau}\|h(j)\|_{F}^2\right)^{1/2}   = \langle \cS, \cS\rangle_{h_2,\tau}^{\frac{1}{2}}. 
	\end{equation}
\end{mydef} 
%In definition~\ref{def:h2innernorm}, the TL inner-product and norm are defined. 
The reader should notice that if $\tau$ goes to infinite, equations~\eqref{eq:TLinprod} and~\eqref{eq:TLnorm} become the classical definition of inner-product and norm for an infinite time horizon for stable systems. However, the TL norm and inner-product are also well defined for unstable systems. Additionally, they can be characterized by matrix equations as it follows. 
\begin{prop}{\bf(TL inner-product and norm characterization)}\label{prop:h2innernormStein} Let $\cS = (A, B, C)$ and $\hat{\cS} = (\hA, \hB, \hC)$ be two LTI discrete-time dynamical systems as in \eqref{eq:LTI}. Then, the $h_2$ TL inner-product can be computed as
\begin{equation}\label{eq:TLinprodStein}
\langle \cS, \hat{\cS} \rangle_{h_2,\tau} = \trace{CY\hat{C}^T} = \trace{B^TZ\hB}, 
\end{equation}
where \[Y = \sum_{j=1}^{\tau} A^{j-1}B\hB^T(\hA^{T})^{j-1}~\,~ \textnormal{and}~\,~ Z = \sum_{j=1}^{\tau} (A^T)^{j-1}C^T\hC\hA^{j-1}.\] Additionally, if $\alpha\beta \neq 1$, for all $\alpha\in \Lambda(A)$ and $\beta \in \Lambda(\hA)$, the matrices $Y$ and $Z$ are the unique solution of the following Stein-like matrix equations
\begin{subequations}\label{eq:SylvesterSteinequations}
\begin{align}
AY\hA^T -Y +B\hB^T - F\hF^T &= 0, \label{eq:SylvesterSteinequationsY}\\ 
A^TZ\hA -Z +C^T\hC - G^T\hG &=0,\label{eq:SylvesterSteinequationsZ}
\end{align}  
\end{subequations}
where $F = A^{\tau}B, \hF = \hA^{\tau}\hB$, $G= CA^{\tau}$ and $\hG = \hC\hA^{\tau}$.
\end{prop}
\begin{proof}  Notice
	$\langle \cS, \hat{\cS} \rangle_{h_2,\tau} = \trace{ C\left(\sum_{j=1}^{\tau}A^{j-1}B\hB^T (\hA^{T})^{j-1}\right)\hC^T} = \trace{CYC^T}. $ Then,  as an application of the telescopic sum on $AY\hA^T - Y$, one obtains that $Y$ satisfies equation \eqref{eq:SylvesterSteinequationsY}. Moreover, equation \eqref{eq:SylvesterSteinequationsY}  has a unique solution if  and only if  $\alpha\beta \neq 1$, for all $\alpha\in \Lambda(A)$ and $\beta \in \Lambda(\hA)$ (see \cite[Theorem 18.2]{dym2013linear}).  The equivalent result for the matrix $Z$ follows similarly.
\end{proof}

Proposition~\ref{prop:h2innernormStein} states that, if the equations  \eqref{eq:SylvesterSteinequations} have unique solutions, then the solutions can be used to compute the TL inner-product via formula \eqref{eq:TLinprodStein}. As a consequence, the TL norm of a system can be computed via
\begin{equation}\label{eq:TLnormStein}
\|\cS\|_{h_2, \tau}^2 = \trace{CP_{\tau}C^T} = \trace{B^TQ_{\tau}B},
\end{equation}
where $P_{\tau}$ and $Q_{\tau}$ are the solutions of~\eqref{eq:FiniteSteinEq1} and~\eqref{eq:FiniteSteinEq2}.

\begin{assumption} From now on, we assume that  $\alpha\beta \neq 1$, for all $\alpha\in \Lambda(A)$ and $\beta \in \Lambda(\hA)$, so that the equations \eqref{eq:SylvesterSteinequations} always have an unique solution. 
\end{assumption}

\subsection{First characterization of error bound}
Let us assume the discrete-time system $\cS = (A, B, C)$ is the full order model and $\hat{\cS} = (\hA,\hB, \hC)$ is the reduced order model. %Hence, for a given order $r$, let us consider the following partition
%\begin{equation*}
%	A = \begin{bmatrix}
%		A_{11} & A_{12}\\
%		A_{21} & A_{22}
%	\end{bmatrix}, B = \begin{bmatrix}
%		B_1 \\ B_2
%	\end{bmatrix},~\textmd{and}~\,C = \begin{bmatrix}
%		C_1 & C_2
%	\end{bmatrix}, 
%\end{equation*}
%where $A_{11} \in \R^{r \times r}$, $B_1 \in \R^{r \times m}$, $C_1 \in \R^{p \times r}$, and the other matrices have appropriate dimensions. The reduced-order model obtained by TL balanced truncation is, then, 
%$\hat{\cS} = (A_{11}, B_1, C_1)$. 
The output of the original system $\cS$ and the reduced system $\hat{\cS}$ can be expressed as 
\[ y(k) = \sum_{j=0}^k h(k-j)u(j),~\textmd{and}~\, \hy(k) = \sum_{j=0}^k \hh(k-j)u(j), \]
where $h(0) =0,$ $h(k) = CA^{k-1}B$, for $k\in \N^*$, is the impulse response of $\cS$, and $\hh(0) =0,$ $\hh(k) = \hC\hA^{k-1}\hB$, for $k\in\N^*$, is the impulse response of $\hat{\cS}$. Hence, the error between $y$ and $\hy$ can be bounded as
\begin{align*}
	\|y(k)-\hy(k)\|_2  &= \left\| \sum_{j=0}^k h(k-j)u(j) - \sum_{j=0}^k \hh(k-j)u(j) \right\|_2 \\
	& \leq \sum_{j=0}^k \left\| \left(h(k-j)-  \hh(k-j)\right) u(j) \right\|_2\\
	& \leq \sum_{j=0}^k \left\| h(k-j)-  \hh(k-j)\right\|_{2} \|u(j) \|_2 \\
	& \leq \sum_{j=0}^k \left\| h(k-j)-  \hh(k-j)\right\|_{F} \|u(j) \|_2, \\
	& \leq \left(\sum_{j=0}^{k} \| h(k) -\hh(k)\|_F^2\right)^{\frac{1}{2}}\left(\sum_{j=0}^{k} \|u(j)\|_2^2\right)^{\frac{1}{2}},
\end{align*}
where in the final step we have applied the Cauchy-Schwarz inequality. By recalling that
\[ \|\cS\|_{h_2,\tau} = \left(\sum_{j=0}^{\tau} \| h(j)\|_{F}^2\right)^{1/2},~\textmd{and}~\,~ \langle \cS, \hat{\cS} \rangle_{h_2,\tau} = \sum_{j=0}^{\tau} \trace{h(j)\hh(j)^T}, \]
one can easily se that
\[ \max_{j=0,1,\dots,\tau}\|y(j)-\hat{y}(j)\|_{2} \leq \left\| \cS-\hat{\cS}\right\|_{h_2,\tau}\left(\sum_{j=0}^{\tau} \|u(j)\|_2^2\right)^{\frac{1}{2}}.  \]
Now, let us first  use the inner-product expression. Hence,
\[\left\| \cS-\hat{\cS}\right\|_{h_2,\tau}^2  = \left\| \cS\right\|_{h_2,\tau}^2 + \left\|\hat{\cS}\right\|_{h_2,\tau}^2 - 2\langle \cS, \hat{\cS} \rangle_{h_2,\tau}. \]
Now, we recall that 
\begin{align*}
	\left\| \cS\right\|_{h_2,\tau}^2 &= \trace{CP_{\tau}C^T} = \trace{B^TQ_{\tau}B},  \\
	\left\| \hat{\cS}\right\|_{h_2,\tau}^2 &= \trace{C_1\hP_{\tau}C_1^T} = \trace{B_1^T\hQ_{\tau}B_1},~\textmd{and} \\
	\langle \cS, \hat{\cS} \rangle_{h_2,\tau} &= \trace{CYC_1^T} = \trace{B^TZB_1}.
\end{align*}
%where $P_\tau, Q_\tau$ are the TL Gramians from the original model \eqref{eq:TLGramians}, $\hP_\tau, \hQ_\tau$ are the TL Gramians of the reduced model \eqref{eq:ROM}, and $\tP_\tau, \tQ_\tau$ are the time-limited Cross-Gramians which are defined as
%\begin{subequations}\label{eq:CrossTLGramians}
%	\begin{align}
%		\tP_{\tau} &= \sum_{k=1}^{\tau-1} A^{k-1}B\left(A_{11}^{k-1}B_1\right)^T, \\
%		\tQ_{\tau} &= \sum_{k=1}^{\tau-1} \left(C_1A_{11}^{k-1}\right)^TCA^{k-1}.
%	\end{align}
%\end{subequations}
%One can easily show the Cross Gramians satisfy the Sylvester type of Stein equations for discrete-time  systems
%\begin{subequations}\label{eq:CrossTLGramians}
%	\begin{align}
%		A\tP_{\tau}A_1^T - \tP_{\tau} +BB_1^T &= A^{\tau}B(A_1^{\tau}B_1)^T  \\
%		A^T\tQ_{\tau}A_1 - \tQ_{\tau} +C^TC_1 &=  \left(CA^{\tau}\right)^TC_1A_1^{\tau}
%	\end{align}
%\end{subequations}
%\red{Maybe I should write this as a Lemma. Moreover there is a condition over the spectrum of those systems to this equation to have unique solution. We need to assume from now one existance of uniquely solution}.
As a consequence, the following error bound result holds.
\begin{prop}\label{prop:FirstErrorBound}The following error bound holds for time-limited balanced truncation of discrete-time systems
	\[ \max_{j=0,1,\dots,\tau}\|y(j)-\hy(j)\|_{2} \leq \epsilon
	\left(\sum_{j=0}^{\tau} \|u(j)\|_2^2\right)^{\frac{1}{2}}, \]
	where 
	\begin{align*} \epsilon^2 &= \trace{CP_{\tau}C^T +C_1\hP_{\tau}C_1^T -2CYC_1^T } \\
		&= \trace{B^TQ_{\tau}B +B_1^T\hQ_{\tau}B_1 - 2B^TZB_1 }
	\end{align*}
	where  $P_{\tau}$ and $Q_{\tau}$ are the TL Gramians of the full order system $\cS$,   $\hat{P}_{\tau}$ and $\hat{Q}_{\tau}$ are the TL Gramians of the reduced order system $\hat{\cS}$, and $Y, Z$ are the solutions of the matrix equations \eqref{eq:SylvesterSteinequationsY} and \eqref{eq:SylvesterSteinequationsZ}. 
\end{prop}

Proposition \ref{prop:FirstErrorBound} provides an error bound for the time-limited norm of the error system $\cS -\hat{\cS}$.  It can be computed in practice by solving two TL Stein equations (as in \eqref{eq:FiniteSteinEq}) for the model $\cS$ and the model $\hat{S}$, and one Stein-like equation (as in \eqref{eq:SylvesterSteinequations}). It is worth noting that this bound is valid for every reduced order model $\hat{S}$. Moreover, it holds even in the case the original model or the reduced order model are unstable.  In the next section, we develop an expression of this error bound tailored for a reduced order model arising from TL balanced truncation.

\section{Output error bound to time-limited balanced truncation}\label{ref:ErrorBoundTLBT}

\subsection{Error bound to TL balanced truncation}
Let suppose that $\cS = (A, B, C)$ is a $n$-order balanced systems associated with the time-limited Gramians $P_{\tau} = Q_{\tau} = \Sigma = \diag{\sigma_1, \dots, \sigma_n}$. Let's consider the following partition
\begin{equation}\label{eq:BalReal}
	A = \begin{bmatrix}
		A_{11} & A_{12} \\ A_{21} & A_{22}
	\end{bmatrix},~\,~B= \begin{bmatrix}
		B_1 \\ B_2 
	\end{bmatrix}~~\, C = \begin{bmatrix}
		C_1 & C_2
	\end{bmatrix}~\,\textnormal{and}~\Sigma = \begin{bmatrix}
		\Sigma_1 & \\ & \Sigma_2
	\end{bmatrix}.
\end{equation}
As a consequence, we must have
\begin{subequations}\label{eq:FullBalGram}
	\begin{align}
		A\Sigma A^T -\Sigma +BB^T-F_{\tau}F_{\tau}^T = 0, \label{eq:FullReachGram}\\
		A^T\Sigma A -\Sigma +C^TC -G^T_{\tau}G_{\tau} = 0,
	\end{align}
\end{subequations}
where 
%\begin{equation}
$F_{\tau} = A^{\tau}B = \begin{bmatrix}
F_{1} \\ F_2
\end{bmatrix}, ~\textnormal{and}~\, G_{\tau} = CA^{\tau} = \begin{bmatrix}
G_1 & G_2
\end{bmatrix}.$
%\end{equation}
The reduced order model obtained by time-limited balanced truncation is  $\hat{\cS} = (\hA, \hB, \hC)$, where $\hA = A_{11} \in \R^{r \times r}$, $\hB = B_1 \in \R^{r \times m}$ and $\hC = C_1 \in \R^{p \times r}$. 

%The error system $\cS_{e} := \cS-\hat{\cS} = (A_{e}, B_{e}, C_e)$ has the following  realisation
%\[ A_{e} = \begin{bmatrix}
%A & \\ & A_{11}
%\end{bmatrix},~\,B_e= \begin{bmatrix}
%B \\  -B_1
%\end{bmatrix}~\textnormal{and}~\, C_e = \begin{bmatrix}
%C & C_1
%\end{bmatrix},  \]
%and it is associated with the Gramians  
%\[ \cP_{e} = \begin{bmatrix}
%\Sigma & -Y \\ 
%-Y^T & \hat{P}
%\end{bmatrix}~\,~\textmd{and}~\,\cQ_{e} = \begin{bmatrix}
%\Sigma & Z \\ Z^T & \hat{Q}_{\tau}
%\end{bmatrix}, \]
%where
%\begin{subequations}
%	\begin{align}
%		AYA_{11}^T -Y +BB_1^T -F_{\tau}\hat{F}_{\tau}^T  &=0,\\
%		A^TZA_{11} - Z +C^TC_1 - G^T_{\tau}\hat{G}_{\tau} &= 0,\label{eq:SylvesterZ} \\
%		A_{11}^T\hQ_{\tau}A_{11}-\hQ_{\tau} +C_1^TC_1 - \hat{G}_{\tau}^T\hat{G}_{\tau} &=0 
%	\end{align}
%\end{subequations}
%where $\hat{F} = A_{11}^{\tau}B_1$ and $\hat{G}_{\tau} = C_1A_{11}^{\tau}$.
Hence, the time-limited $h_2$ norm of the error system is 
\begin{align} \|\cS_{e}\|_{h_{2,\tau}}^2 &= \trace{B^T\Sigma B - 2B^TZB_1 + B_1\hat{Q}_{\tau}B_1} \nonumber \\
	&= \trace{B^T\Sigma B - 2B_1^TZ_1B_1-2B_2^TZ_2B_1 + B_1\hat{Q}_{\tau}B_1}. \label{eq:FirstNormDev}
\end{align}
By developing the term (2,1) of \eqref{eq:FullReachGram},  we obtain
\[A_{11}\Sigma_1A_{21}^T + A_{12}\Sigma_2A_{22}^T + B_1B_2^T - F_1F_2^T = 0,\]
and consequently 
\[\trace{-2B_2^TZ_2B_1} = \trace{-2B_1B_2^TZ_2} = \trace{2A_{11}\Sigma_1A_{21}^TZ_2+2A_{12} \Sigma_2A_{22}^TZ_2-2F_1F_2^TZ_2}. \]
Substituting in \eqref{eq:FirstNormDev}, yields
\begin{align*}
	\|\cS_{e}\|_{h_{2,\tau}}^2 = \trace{B^T\Sigma B - 2B_1^TZ_1B_1 + 2A_{11}\Sigma_1A_{21}^TZ_2 +2A_{12}\Sigma_2A_{22}^TZ_2 - 2F_1F_2^TZ_2 +B_1^T\hat{Q}B_1}.
\end{align*}
For developing the term $\trace{2A_{11}\Sigma_1A_{21}^TZ_2}$, consider  the entry $(1,1)$  of \eqref{eq:SylvesterSteinequationsZ}: 
\[A_{11}^TZ_1A_{11}+A_{21}^TZ_2A_{21}-Z_1+C_1^TC_1 -G_{1}^T\hat{G} \]
%and consequently 
leading to
\begin{align*} 
	\trace{2A_{11}\Sigma_1A_{21}^TZ_2} &= \trace{2\Sigma_1A_{21}^TZ_2A_{11}} \\
	&= \trace{2\Sigma_1Z_1 - 2\Sigma_1A_{11}^TZ_1A_{11} - 2\Sigma_1C_1^TC_1+2\Sigma_1G_1^T\hat{G}}.
\end{align*}
Hence, 
\begin{align*}
	\|\cS_{e}\|_{h_{2,\tau}}^2 = &\trace{B^T\Sigma B - 2B_1^TZ_1B_1 + 2\Sigma_1Z_1 +2\Sigma_1G_1^T\hat{G} -2\Sigma_1A_{11}^TZ_1A_{11}} \\
	&+\trace{-2\Sigma_1C_1^TC_1+2A_{12}\Sigma_2A_{22}^TZ_2 - 2F_1F_2^TZ_2 +B_1^T\hat{Q}_{\tau}B_1}. 
\end{align*}
From now the steps get particularly different from derivations for TLBT for continuous-time systems, because, for discrete-time systems, the reduced order model is not balanced.
Recalling that
\[ \trace{B^T\Sigma B} = \trace{C\Sigma C^T}, ~\,\textmd{and}~\,~ \trace{B_1^T\hQ_{\tau} B_1} = \trace{C_1\hP_{\tau} C_1^T}, \]
gives
\begin{align*}
	\|\cS_{e}\|_{h_{2,\tau}}^2 &= \trace{2A_{12}\Sigma_2A_{22}^TZ_2 +C_2\Sigma_2C_2^T -C_1\Sigma_1C_1^T + C_1\hat{P}_{\tau}C_1^T} \\ 
	&+\trace{-2B_1^TZ_1B_1-2A_{11}\Sigma_1A_{11}^TZ_1 +2\Sigma_1Z_1} \\
	&+\trace{2\Sigma_1G_1^T\hat{G}-2F_1F_2^TZ_2}.
\end{align*}
Since
\[A_{11}\Sigma_1A_{11}^T +A_{12}\Sigma_2A_{12}^T - \Sigma_1 B_1B_1^T - F_1F_1^T =0 \]
it holds
\[\trace{-2B_1^TZ_1B_1-2A_{11}\Sigma_1A_{11}^TZ_1 +2\Sigma_1Z_1}= \trace{2A_{12}\Sigma_2 A_{12}^TZ_1 - 2F_1F_1^TZ_1}.  \]
Summarizing all of these steps together, we have the following theorem.

\begin{theo}\label{theo:BTh2TLbounds} Let $S = \left(\begin{bmatrix} A_{11} & A_{12} \\
	A_{21} & A_{22}
	\end{bmatrix}, \begin{bmatrix}
	B_1 \\  B_2
	\end{bmatrix}, \begin{bmatrix}
	C_1 & C_2
	\end{bmatrix}\right)$ be a balanced system and $\hat{\cS} = (A_{11}, B_1, C_1)$ be the $r$ order reduced model obtained by time-limited balanced truncation. The time-limited $h_2$ norm of the error system is given by
	\begin{equation}\label{eq:LTLBTerror}
		\begin{array}{rcl}
			\|\cS_{e}\|_{h_{2,\tau}}^2 &=& \trace{C_2\Sigma_2C_2^T+ 2A_{12}\Sigma_2A_{:2}^TZ}  +\trace{C_1(\hP_{\tau}-\Sigma_1)C_1^T} \\
			&&+ 2\trace{\Sigma_1G_1^T\hat{G} - F_1F^TZ}, \\
			&=&  \trace{B_2^T\Sigma_2B_2+ 2A_{21}^T\Sigma_2A_{2:}Y}  +\trace{B_1^T(\hQ_{\tau}-\Sigma_1)B_1} \\
			&&+ 2\trace{\Sigma_1F_1^T\hat{F} - G_1G^TY},
		\end{array}
	\end{equation}
	where $A_{:2} = \begin{bmatrix}
	A_{12} \\ A_{22}
	\end{bmatrix},$   $A_{2:} = \begin{bmatrix}
	A_{21} & A_{22}
	\end{bmatrix}$, $F = A^{\tau}B =  \begin{bmatrix}
	F_1 \\ F_2
	\end{bmatrix}$, $G = CA^{\tau} = \begin{bmatrix}
	G_1 & G_2 
	\end{bmatrix}$ and $\hat{G} = CA_{11}^{\tau}$
\end{theo}
Theorem~\ref{theo:BTh2TLbounds} gives an analytic expression to the error bound provided in Proposition~\ref{prop:FirstErrorBound} in the case the reduced order model is obtained by TL balanced truncation. This characterization highlights how the error bound depends on the singular values $\Sigma$ and the time-limited terms $G, \hat{G}, F, \hat{F}$. It should be emphasized that it holds even if the original and reduced order models are unstable, provided the solvability conditions for the involved matrix equations hold. This expression depends on the partition matrix of the balanced full order model, the partitioning of the time-limited Hankel singular values matrices, and the matrices $Y$ and $Z$ appearing in Proposition~\ref{prop:h2innernormStein} for the computation of the inner product.  Readers should notice that the TL error bound differs for the infinite horizon error bound in~\eqref{eq:InfHorErrorBound} from the residual time-limited term 
\begin{equation}\label{eq:TLresidual}
R_{\tau} := 2\trace{\Sigma_1G_1^T\hat{G} - F_1F^TZ}.
\end{equation} 
As one would expect, we will see that if we take $\tau \rightarrow
\infty$, then $R_{\tau} \rightarrow 0$, and the expression given in \eqref{eq:LTLBTerror} will tend to the error bound expression for infinite horizon. In what follows, we will study the impact of TL terms $G, F, \hG$ and $\hF$  in the error bound. 

\subsection{Time-limited residue impact in error bound}
\subsubsection{Stability preservation}
For the infinite horizon case, balanced truncation for discrete-time systems always produce a stable reduced order model which is not automatically the case for the time-limited variant. In what follows we provide a sufficient condition to the reduced order model obtained by TLBT to be stable.  We keep the notation used in last section. 
\begin{prop}{(\bf Stability preservation)}\label{prop:StabilCond} Suppose that 
\[Q = A_{12}\Sigma_2A_{12}^T +B_1B_1^T - F_1F_1^T\geq 0, \] 
and the pair $(A_{11}, Q)$ is controllable.
Then the reduced order model is stable.
\end{prop}
\begin{proof}
From the Stein equation~\eqref{eq:FullReachGram} it follows
\begin{equation}\label{eq:SteinStab}
A_{11}\Sigma_1A_{11}^T + A_{12}\Sigma_2A_{12}^T-\Sigma_1 + B_1B_1^T - F_1F_1^T = 0.
\end{equation}
Let $v \in \C^r$ and $\mu\in \C$ be an eigenpair of $A_{11}^T$, \emph{i.e.}, $A_{11}^Tv = \mu v$. Then, multiply \eqref{eq:SteinStab} by $v^*$ (on the left),
and $v$ (on the right) to obtain
\[ (1-|\mu|^2)v^*\Sigma_1 v = v^*\underbrace{\left(A_{12}\Sigma_2A_{12}^T +B_1B_1^T - F_1F_1^T\right)}_{=Q}v \geq 0.\]  
Hence, since $\Sigma_1>0$, this immediately implies $|\mu|\leq 1$. 
Now assume that $|\mu| = 1$.  In this case, we have $v^*Q = 0$. Moreover, if we multiply \eqref{eq:SteinStab} by $A_{11}$ (by the right) and $A_{11}^T$ (by the left), we obtain
\[ A_{11}^2\Sigma_1(A_{11}^T)^2 -A_{11}\Sigma_1A_{11}^T +A_{11}QA_{11}^T = 0. \] Hence, if we multiply, the later equation by $v^*$ and $v$, we obtain
\[ 0=(|\mu|^2-|\mu|^4)v^*\Sigma_1 v  = v^*A_{11}QA_{11}^Tv .  \]
As a consequence, we have $v^*A_{11}Q$. By induction, we conclude that $v^*A^{k-1}_{11}Q = 0$ for $k>0$, which implies that the pair $(A,Q)$ is not reachable. Then $|\mu|<1$ and the matrix $A_{11}$ is stable.
\end{proof}

Proposition \ref{prop:StabilCond} gives a sufficient condition for the ROM produced by TL balanced truncation to be stable. It is worth mentioning that this condition relies on the matrices \eqref{eq:BalReal} of the balanced realization.

\subsubsection{Asymptotic behavior of $A^p$}
 Given matrices $A$ and $A_{11}$, there exist constants $c, \hat{c}>0$ and $\lambda$, $\hat{\lambda} \in \C$ such that
\begin{equation}\label{eq:AssympProp}
\|A^{p}\|_2 \leq c\cdot \lambda^{p}~\,~\textnormal{and}~\,~  \|A_{11}^{p}\|_2 \leq \hc \cdot \hat{\lambda}^{p}
\end{equation}
for all $p\in \N$ and for any matrix norm $\|\cdot\|$. Moreover, if  $\Lambda(A)$ and $\Lambda(A_{11})$ lies inside the open unit disc, \emph{i.e.}, $A$ and $A_{11}$ are stable matrices, then $\lambda$ and $\hat{\lambda}$ can be chosen such that $|\lambda|<1$ and $|\hat{\lambda}|<1$. If  $A, A_{11}$ are assumed to be stable matrices, we know that $A^k \rightarrow 0$ and $A_{11}^{k} \rightarrow 0$ whenever $k \rightarrow \infty$.  Equation~\eqref{eq:AssympProp} describes the asymptotic behavior of the  norm $\|\cdot \|$  of those power matrices, \emph{i.e.},  how fast those sequence  of matrices goes to zero.  

There are different ways to compute $c$, $\hc$, $\hat{\lambda}$ and $\lambda$. For example, in the case where $\|\cdot \|$ is the $p$ induced norm and $A$ is diagonalizable, \emph{i.e.},  $A = XD X^{-1}$ with $X$ nonsingular and $D$ diagonal, we can choose $\lambda = \rho(A)$, the spectral radius of $A$, and $c = \kappa(X) = \|X\|_p\|X^{-1}\|_p$ is the condition number of $X$ in the norm $\|\cdot\|_p$. We refer to \cite{higham1995matrix} for other asymptotic bounds of the form \eqref{eq:AssympProp}. Additionally, the recent paper \cite{crouzeix2017numerical} provides a new improvement  on the bounds of matrix functions, which 
includes power matrices. The main result of \cite{crouzeix2017numerical}  states that
\[ \|f(A)\|_2 \leq (1+\sqrt{2}) \sup_{z\in \Omega} |f(z)|,  \]
where $\Omega = \{ z \in \C, z = v^HAv, ~\,~\textnormal{for all}~\,~v\in \C^n, \|v\| = 1\}$ is the numerical range of the matrix $A\in \C^{n \times n}$ (also called field of values). % and $\|\cdot\|$ is the operator norm. 
Hence, with $f(z)=z^{\tau}$, the numerical radius $\lambda=r(A):=\max_{z\in\Omega}|z|$, $c=1+\sqrt{2}$, we can always bound
\begin{equation}\label{eq:FOVconst} 
\|A^{\tau}\| \leq (1+\sqrt{2})\cdot \lambda^{\tau}
\end{equation}
because $r(A^{\tau})\leq r(A)^{\tau}$. 
Since in our case, $\tau<\infty$, the above bounds will always be finite even if spectrum or numerical range do not lie inside the unit disc. 
%if $\Omega$ lies inside of the open unitary disc, $f(z) = z^p$ and $\lambda = \sup_{z\in \Omega} |z|,$ then $\lambda<1$ and
%As a consequence, whenever the numerical range of a matrix $A$ lies inside the open unit disc, the parameters of \eqref{eq:AssympProp} can be choose as $\lambda = \alpha$ 
%%\sup_{z\in \Omega} |z|$ 
%and $c = (1+\sqrt{2})$. 
From now one we assume that such $c$, $\hc$, $\hat{\lambda}$, and $\lambda$ as in \eqref{eq:AssympProp} are available. 
	
 \subsubsection{Asymptotic impact of  TL residue}
Let us now discuss the impact of $R_{\tau}$ from equation~\eqref{eq:TLresidual} in the error bound of Theorem~\ref{theo:BTh2TLbounds}. The terms of $R_{\tau}$ can be bounded as
\begin{align*}
2\trace{\Sigma_1G_1^T\hG} &\leq \|\Sigma_1\|_F\|G_1\|_F\|\hG\|_F,   \\
2\trace{F_1F^TZ} &\leq \|Z\|_F\|F_1\|_F\|F\|_F.   
\end{align*}
We recall that, if $V^T = \begin{bmatrix}
I_{r} & 0_{r \times (n-r)}
\end{bmatrix}$, then $F = A^{\tau}B$, $F_1 = V^TA^{\tau}B$, $G_1 = CA^{\tau}V$ and $\hG = C_1A_{11}^{\tau}$. Additionally , the norms of $\|F_1\|_F$, $\|F\|_F $ are bounded by $c\lambda^{\tau}\|B\|_F$,  $\|G_1\|_F$ is bounded by $c\lambda^{\tau}\|C\|_F$,  and  $\|\hG\|_F$ is bound by  $\hat{c}\hat{\lambda}^{\tau}\|C_1\|_F$, where  and $\lambda,\hat{\lambda}, c, \hat{c}>0$ are suitable constants. Moreover, if we assume that $\Lambda(A)$ and $\Lambda(A_{11})$ lies inside the open unit disc, the norms decay fast whenever the value of $\tau$ increases and the term $R_{\tau} \rightarrow 0$ if $\tau$ goes to infinity. As  a consequence, the error bound formulas provided in Theorem~\ref{theo:BTh2TLbounds} coincide with those for the infinite horizon (see equation~\eqref{eq:InfHorErrorBound}) in the limit $\tau \rightarrow \infty$. 
\begin{remark} For the infinite horizon case, where the original  and  the reduced order model are stable, the error bound in \eqref{eq:InfHorErrorBound} can be bounded by 
\begin{equation}\label{eq:UpperBoundInfHorizon}
\|\cS -\hat{\cS}\|_{h_2} \leq \trace{C_2\Sigma_{2,\infty}C_2^T+ 2A_{12}\Sigma_{2,\infty}A^{T}_{:2}Z_{\infty}},
\end{equation}
because the term $\trace{C_1(\hat P_{\infty} -\Sigma_{1,\infty})C_1^T}\leq 0.$ Indeed, the matrix $E_{\infty} = P_{\infty} -\Sigma_{1,\infty}$ is negative definite, since it satisfies the following  Stein equation
\[A_{11}E_{\infty}A_{11}^T - E_{\infty} -A_{12}\Sigma_{2,\infty}A_{12}^T =0, \]
and $-A_{12}\Sigma_{2,\infty}A_{12}^T$ is a negative semi definite matrix. As a consequence, in Equation \eqref{eq:UpperBoundInfHorizon}, we can see explicitly that the decay of singular values will lead to a decay in the error for the infinite horizon case. We believe that this expression is new and it was not present in \cite{chahlaoui2012posteriori}. 
\end{remark}

%Moreover, if we define $E = \hat{P}-\Sigma_1$, 
 \subsubsection{Error bound depending on $\Sigma_2$ and asymptotic parameters}

Now, we wish to explicitly describe the dependency of expression \eqref{eq:LTLBTerror} on the neglected singular values $\Sigma_2$ and on the time-limited terms $F$, $\hat{F}$, $G$, and $\hat{G}$. From now one, we will  assume that $A$ and $A_{11}$ are stable, \emph{i.e.}, that their eigenvalues lie inside the unitary open disc. Additionally, we assume that $|\lambda|<1$ and $\hat{\lambda}<1$.   We will discuss the case where  $A$ or $A_{11}$ are  unstable in Remark~\ref{remark:unstable}.   

Let us first write $E = \hP_{\tau} - \Sigma_1$. As a consequence, $E$ satisfies the following Stein equation
\[A_{11}EA_{11}^T - E -A_{12}\Sigma_2A_{12}^T +F_1F_1^T-\hF\hF^T = 0.  \]
Consider the composition $E = E_{\Sigma_2} + E_{TL}$, where
\begin{align*}
A_{11}E_{\Sigma_2}A_{11}^T - E_{\Sigma_2} -A_{12}\Sigma_2A_{12}^T  &= 0, \\
A_{11}E_{TL}A_{11}^T - E_{TL} +F_1F_1^T-\hF\hF^T &= 0.
\end{align*}
Since $A_{11}$ is stable and $-A_{12}\Sigma_2A_{12}^T$ is a symmetric negative semidefinite matrix, $E_{\Sigma_2}$ is also symmetric negative semidefinite. As a consequence, we can rewrite the term $\trace{C_1(\hP_{\tau}-\Sigma_1)C_1^T}$ as
\begin{equation}\label{eq:BoundE}
\trace{C_1(\hP_{\tau}-\Sigma_1)C_1^T} = \trace{C_1(E_{\Sigma_2} +E_{TL})C_1^T} \leq  \trace{C_1(E_{TL})C_1^T}.
\end{equation}
%because $E_{\Sigma_2}$ is negative semidefinite. 
Since 
$A_{11}E_{TL}A_{11}^T  +F_1F_1^T-\hF\hF^T = E_{TL}$ and $A_{11}$ is stable, $E_{TL}$ can be written as the following infinite series
\begin{equation}\label{eq:InfSum1}
E_{TL} = \ds \sum_{j=1}^{\infty} A_{11}^{j-1}\cF_{TL} (A_{11}^T)^{j-1},~\,\textnormal{with}~\,\cF_{TL} = F_1F_1^T-\hF\hF^T. 
\end{equation}
Consequently,
\begin{align*}
\|E_{TL}\|_2  &\leq \sum_{j=1}^{\infty} \|A_{11}^{j-1}\|_2\|\cF_{TL}\|_2\|(A_{11}^T)^{j-1}\|_2 \\
&\leq \|\cF_{TL}\|_2  \sum_{j=1}^{\infty} \hc^2\cdot(\hat{\lambda}^{j-1})^2 = \|\cF_{TL}\|_2   \frac{\hc^2}{1-\hat{\lambda}^2}.
\end{align*}
Using similar steps one can show that
\begin{align}\label{eq:InfSum2}
	\|Z\|_2 \leq \frac{c\cdot \hc}{1-\lambda\hat{\lambda} }\|\cM\|_2, ~\textnormal{with} ~\, \cM = C^TC_1 - G^T\hG.
\end{align}
Finally, we can bound
\begin{align*}
\left|\trace{C_1(\hP_{\tau}-\Sigma_1)C_1^T}\right|  	&\leq p\|C_1\|_2^2\|E_{TL}\|_2 \leq  \frac{p\cdot\hc^2}{1-\hat{\lambda}^2}\|C_1\|_2^2\|\cF_{TL}\|_2, \\
\left|\trace{2A_{12}\Sigma_2A_{:2}^TZ}\right|           	&\leq 2r\|A_{12}\|_2\|\Sigma_2\|_2\|A_{:2}\|_2\|Z\|_2 \\ &\leq  \sigma_{r+1} \frac{2r\cdot c\cdot \hc}{1-\lambda\hat{\lambda} }\|A_{12}\|_2\|A_{:2}\|_2\|\cM\|_2  \\
|\trace{C_2\Sigma_2C_2^T}| 											  &\leq    p\|C_2\|_2^2\|\Sigma_2\|_2 = p\|C_2\|_2^2 \sigma_{r+1}, \\
|2\trace{\Sigma_1G_1^T\hat{G}}| 								   & \leq 2p\cdot\sigma_1\|G_1\|_2\|\hat{G}\|_2, \\
|2\trace{F_1F^TZ}| 															   &\leq \frac{2m\cdot c\cdot \hc}{1-\lambda\hat{\lambda} }\|F_1\|_2\|F\|\|\cM\|_2,
\end{align*}

Additionally, using \eqref{eq:AssympProp}, we have
\begin{align*}
\|\cF_{TL}\|_2 &\leq \|F_1\|_2^2+ \|\hF\|_2^2 \leq c^2\lambda^{2\tau}\|B\|_2 + \hc^2\hat{\lambda}^{2\tau}\|B_1\|_2, \\
\|\cM\|_2 &\leq \|C\|_2\|C_1\|_2 + \|G\|_2\|\hG\|_2\leq \|C\|_2\|C_1\|_2(1+c\hc \lambda^{\tau} \hat{\lambda}^{\tau}). \\
\end{align*} 
The following theorem assembles all these results.
\begin{theo}\label{theo:SepBounds}  Let $S = \left(\begin{bmatrix} A_{11} & A_{12} \\
	A_{21} & A_{22}
	\end{bmatrix}, \begin{bmatrix}
	B_1 \\  B_2
	\end{bmatrix}, \begin{bmatrix}
	C_1 & C_2
	\end{bmatrix}\right)$ be a balanced system and $\hat{\cS} = (A_{11}, B_1, C_1)$ be the $r$ order reduced model obtained by time-limited balanced truncation. Assume that $\lambda<1$ and $\hat{\lambda}<1$. Then the following bound holds.
	\begin{equation}\label{eq:ErrorBound_Assymp}
	\|S_e\|_{h_{2,\tau}}^2 \leq J(\tau) \cdot\sigma_{r+1} +  J_{TL}(\tau), 
	\end{equation}
	where
	 \begin{align*}
	  	 	J(\tau) &= p\|C_2\|_2^2+ 	\frac{2r c \hc (1+c\hc \lambda^{\tau} \hat{\lambda}^{\tau})}{1-\lambda\hat{\lambda} }\|A_{12}\|_2\|A_{:2}\|_2\|C\|_2\|C_1\|_2, \\
	  	 	J_{TL}(\tau) &= \frac{p\cdot\hc^2}{1-\hat{\lambda}^2}\|C_1\|_2^2(c^2\lambda^{2\tau}\|B\|_2 + \hc^2\hat{\lambda}^{2\tau}\|B_1\|_2) + 2p\cdot\sigma_1c\hc \lambda^{\tau} \hat{\lambda}^\tau \|C\|_2\|C_1\|_2 \\ &+ \frac{2m\cdot c\cdot \hc}{1-\lambda\hat{\lambda} }c^2\lambda^{2\tau}\|B\|_2^2\|C\|_2\|C_1\|_2(1+c\hc \lambda^{\tau} \hat{\lambda}^{\tau}).
	 \end{align*} 
\end{theo}
Theorem~\ref{theo:SepBounds} splits the bounds from~\eqref{eq:LTLBTerror} into  $J(\tau) \sigma_{r+1}$ and $J_{TL}(\tau)$. The term $J(\tau) \sigma_{r+1}$ depends linearly on $\sigma_{r+1}$, \emph{i.e.}, the largest neglected Hankel singular value. The term $J_{TL}(\tau)$ represents the time-limited terms. If $\tau$ goes to infinite we have
\[J_{TL}(\tau) \rightarrow 0 ~\textnormal{and}~\,~ J(\tau) \rightarrow J_{\infty} =  p\|C_2\|_2^2+ 	\frac{2r c \hc}{1-\lambda\hat{\lambda} }\|A_{12}\|_2\|A_{:2}\|_2\|C\|_2\|C_1\|_2.\]

%\red{Should we add something like: under the stronger assumption that the field of values $\Omega(A),~\Omega(A_{11})$ are contained in the unit disk (needed for
%both?), we can set
%$c=\hc=1+\sqrt{2}$ and $\lambda=\sup_{z\in \Omega(A)} |z^p|<1$, %$\hat{\lambda}=\sup_{z\in \Omega(A_{11})} |z^p|<1$ the result becomes $\ldots$ ??\\
%\red{Also, which of the prior bounds / estimations fall apart if $A_{11}$ is unstable? Could we still provide some bound? $\sup_{z\in \Omega(A_{11})} |z^p|$ woulds
%still be finite.}
\begin{remark}\label{remark:unstable} In the case $\lambda \geq 1$, or $\hat{\lambda}\geq 1$ , the equations \eqref{eq:BoundE}, \eqref{eq:InfSum1} and $\eqref{eq:InfSum2}$ do not hold anymore. As a consequence, Theorem\ref{theo:SepBounds} is not valid in this form.
However, we can still bound the terms 
\begin{align*}
\left|\trace{C_1(\hP_{\tau}-\Sigma_1)C_1^T}\right|  &\leq p\|C_1\|_2^2\|\hP_{\tau}-\Sigma_1\|_2, \\
\left|\trace{2A_{12}\Sigma_2A_{:2}^TZ}\right|           	&\leq 2r\sigma_{r+1}\|A_{12}\|_2\|A_{:2}\|_2\|Z\|_2, \\
|2\trace{F_1F^TZ}| 	&\leq 2m \|Z\|_2\|F_1\|_2\|F\|_2.
\end{align*}
Hence, the equivalent to Theorem~\ref{theo:SepBounds} has explicit dependencies on $Z$, $\hP_{\tau}$ and $\Sigma_1$. 
\end{remark}
\begin{remark}
 For generalized systems
 \begin{equation}\label{eq:gLTI}
\begin{array}{rl}
		Mx(k+1) &= Ax(k) +Bu(k),~\textrm{for}~\,k\in \N =\{0, 1, 2, \dots\}\\
		y(k) &= Cx(k) ,~\,~x(0) = x_0,
	\end{array}
\end{equation}
with a nonsingular matrix $M\in\Rnn$, the results established so far holds as well with minor modifications that we give next without derivations as those follow the same lines as in the continuous-time situation~\cite{morKue18,morRedK18}.
In particular, the time-limited Gramians  are $P_{\tau}$, $M^TQ_{\tau}M$ and are now the obtained from the solutions of the generalized Stein equations
\begin{subequations}\label{eq:genStein}
\begin{align}
AP_{\tau}A^T - MP_{\tau}M^T +BB^T &= F_MF_M^T,\quad F_M:=M(AM^{-1})^{\tau}B \label{eq:FiniteGenSteinEq1} \\
		A^TQ_{\tau}A^T - M^TQ_{\tau}M + C^TC &= G_M^TG_M, \quad G_M:=C(M^{-1}A)^{\tau}\label{eq:FiniteGenSteinEq2}
\end{align}
\end{subequations}
 Obviously, the infinite Gramians of~\eqref{eq:gLTI} are given by omitting the terms $F_M$, $G_M$ in~\eqref{eq:genStein}.
Since in balanced coordinates $M$ is transformed to the identity and  $M_{11}=I_r$, the matrix equations for Gramians $\hat{P}_{\tau}$ and $\hat{Q}_{\tau}$ of the reduced system remain unchanced. The Sylvester equations~\eqref{eq:SylvesterSteinequations} transform to
   \begin{subequations}\label{eq:GenSylvesterSteinequations}
\begin{align}
AY\hA^T -MY +B\hB^T - F_M\hF^T &= 0, \label{eq:GenSylvesterSteinequationsY}\\ 
A^TZ\hA -M^TZ +C^T\hC - G_M^T\hG &=0,\label{eq:GenSylvesterSteinequationsZ}.
\end{align}  
\end{subequations}
Consequently, by using the adapted Gramians and matrix equations, the error bounds still hold.  
\end{remark}
%we have
%\begin{align*}
%\|E_{TL}\|  &= \|A_{11}E_{TL}A_{11}^T  +F_1F_1^T-\hF\hF^T\|\\
%				  &\leq 	\|A_{11}\|\|A_{11}^T\|\|E_{TL}\| +\|F_1F_1^T-\hF\hF^T\|.
%\end{align*}
%Hence,
%\[ \|E_{TL}\|  \leq \displaystyle\frac{\|F_1F_1^T-\hF\hF^T\|}{1-\|A_{11}\|^2}   \]

%$E_{TL}$ can be written as 
%\begin{equation*}\label{eq:ETL}
%E_{TL} = \sum_{k=1}^{\infty} A_{11}^{k-1}\left(F_1F_1^T-\hF\hF^T\right)\left(A_{11}^T\right)^{k-1}.
%\end{equation*} 
%As a consequence,
%\[  E_{TL} = \sum_{k=1}^{\infty} A_{11}^{k-1}\left(F_1F_1^T-\hF\hF^T\right)\left(A_{11}^T\right)^{k-1}.\]
%\[ \|E_{TL}\|_{2} \leq c_{A_{11}}\|F_1F_1^T-\hF\hF^T\|_{2} ,\]
%with $c_{A_{11}} = \sum_{k=1}^{\infty} $
\subsection{Small-scale example}
Now we illustrate the obtained results by applying BT and TLBT to a small-scale system and computing the infinite horizon (equation \eqref{eq:InfHorErrorBound}) and time-limited bounds (Proposition \ref{prop:FirstErrorBound}  or Theorem~\ref{theo:BTh2TLbounds}). For this, we consider a random stable single-input single-output (SISO) system of order $n = 10$, generated by  the \textsc{Matlab}\xspace command \textsf{rss} and converted to discrete-time system using  a zero-order hold procedure (command \textsf{c2d}) with discretization step d$t = 1$sec. We considered the horizon of $\tau = 20$. Then the infinite horizon and time-limited Gramians and error bounds are computed using \textsc{Matlab}\xspace direct solver (command \textsf{dlyap}). Finally, two reduced models of order $r=6$ are computed using BT and TLBT and, in this case, the two models are stable.

We compare the time domain response of  the corresponding two reduced models. For this, we use as the control input $u(1) = 1, u(k) =0,$ for $k>1$. The results of the absolute errors are depicted in Figure~\ref{fig:SmallArticleEx}, as well as the  bounds from Equation \eqref{eq:InfHorErrorBound} and Proposition \ref{prop:FirstErrorBound} for BT and TLBT, and twice the sum of the neglected (time-limited) Hankel singular values $\boldsymbol{\sigma}_r$ for (TL)BT. By inspecting the time-domain error between the original response and the two reduced-order models, we observe that the TLBT generally produces better compared to BT in the given time-limited interval, as expected. Additionally, the bounds from Equation \eqref{eq:InfHorErrorBound} and Proposition \ref{prop:FirstErrorBound} are satisfied by the errors (see Table~\ref{tab:SmallExample_bounds} for the numerical values). 

Now, we compute also the bounds from Theorem~\ref{theo:SepBounds} for TLBT. To this aim, first we need an estimation of the constants $c, \hat{c}, \lambda, \hat{\lambda}$. We considered two sets of constants, one obtained using the eigenvalue decomposition and another using the field of values and the inequality \eqref{eq:FOVconst}. Those values and the error bounds are displayed  in the Table~\ref{tab:SmallExample_Asymp}. Notice that for the values related to the eigenvalue decomposition, we have that $\lambda<1$ and $\hat{\lambda}<1$. As a consequence, Theorem~\ref{theo:SepBounds} holds, and we use it to compute the error bounds displayed. However, for the  field of values, we have that $\lambda >1$, and so Theorem~\ref{theo:SepBounds} does not hold anymore. To circumvent the issue, we use the ideas in Remark \ref{remark:unstable} to compute  the bound display.  By inspecting Tables \ref{tab:SmallExample_bounds} and \ref{tab:SmallExample_Asymp}, we conclude  that the bounds depending on the asymptotic parameters are less sharp. Indeed, they were developed  in order to study the asymptotic behavior of the error with respect to $\tau$, and  their value is theoretical rather than practical.

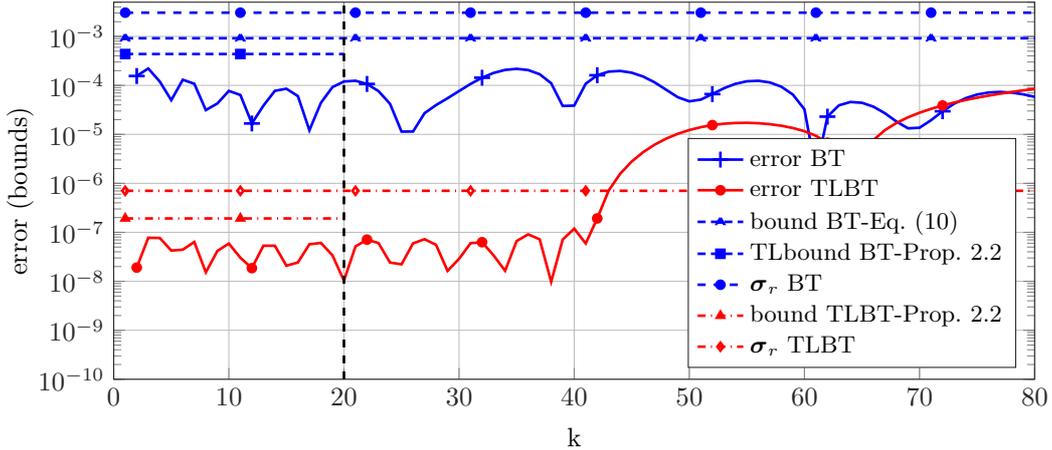
\begin{figure}[t]
		\begin{center}
	% This file was created by matlab2tikz v0.4.7 running on MATLAB 9.1.
% Copyright (c) 2008--2014, Nico Schlömer <nico.schloemer@gmail.com>
% All rights reserved.
% Minimal pgfplots version: 1.3
% 
% The latest updates can be retrieved from
%   http://www.mathworks.com/matlabcentral/fileexchange/22022-matlab2tikz
% where you can also make suggestions and rate matlab2tikz.
% 
\begin{tikzpicture}
\pgfplotstableread{SimpleBounds_Order10.dat}\loadedtable
 \pgfplotscreateplotcyclelist{simple_bound}{%
	{color=blue, solid, mark=+,line width=1pt, mark options={solid,mark size=3pt},mark repeat=10},
	{color=red,solid,mark=*,mark size=1.5pt,line width=1pt, mark repeat=10},
	{color=blue,densely dashed,mark=triangle*,mark size=1.5pt,line width=1pt, mark repeat=10},
	{color=blue,densely dashed,mark=square*,mark options={solid,mark size=1.5pt},line width=1pt, mark repeat=10},
	{color=blue, dashed,mark=*,mark options={solid,mark size=1.5pt},line width=1pt, mark repeat=10},
	{color=red, dashdotted, mark=triangle*, mark options={solid,mark size=1.5pt},line width=1pt, mark repeat=10},
	{color=red, dashdotted, mark=diamond, mark options={solid,mark size=1.5pt},line width=1pt, mark repeat=10},
}
\begin{axis}[%
cycle list name=simple_bound,
width=0.951\wex,
height=\hex, %3.565625in,
scale only axis,
%separate axis lines,
every outer x axis line/.append style={white!15!black},
every x tick label/.append style={font=\color{white!15!black}},
xmin=0,
xmax=80,
xlabel={k},
every outer y axis line/.append style={white!15!black},
every y tick label/.append style={font=\color{white!15!black}},
ymode=log,
ymin=1e-10,
ymax=5*1e-3,
yminorticks=true,
xmajorgrids,
ymajorgrids,
ylabel={error (bounds)},
legend style={draw=white!15!black,fill=white,legend cell align=left,at={(0.98,0.03)},font=\small, anchor=south east, align=left},
legend entries={error BT, error TLBT, bound BT-Eq.~\eqref{eq:InfHorErrorBound}, TLbound BT-Prop.~\ref{prop:FirstErrorBound}, $\boldsymbol{\sigma}_r$ BT, bound TLBT-Prop.~\ref{prop:FirstErrorBound}, $\boldsymbol{\sigma}_r$ TLBT}
]
\addplot table[x index=0,y index=2] {\loadedtable};

\addplot table[x index=0,y index=1] {\loadedtable};

\addplot table[x index=0,y index=4] {\loadedtable};

\addplot table[x index=0,y index=6] {\loadedtable};

\addplot table[x index=0,y index=7] {\loadedtable};

\addplot table[x index=0,y index=3] {\loadedtable};

\addplot table[x index=0,y index=5] {\loadedtable};

\addplot [color=black,line width=1.pt,dashed,forget plot]
table[row sep=crcr]{20	1e-10\\
	20	5e-3\\
};
\end{axis}
\end{tikzpicture}%
	\caption{Output errors $|y(k)-\hy(k)|$, error bounds, and sum of neglected HSVs $\boldsymbol{\sigma}_r$ for (TL)BT reduction of small scale example to order $n=10$, reduced order $r =6$ with time limit $\tau=20$.}\label{fig:SmallArticleEx}
		\end{center}
\end{figure}

\begin{table}[t]
	% \footnotesize
	\centering
	\caption{Summary of the error bounds for small scale example}\footnotesize
	\begin{tabularx}{0.82\textwidth}{|X|l|l|l|l|}
		 \hline
		 Eq. \eqref{eq:InfHorErrorBound} for BT & Prop. \ref{prop:FirstErrorBound} for BT& Prop. \ref{prop:FirstErrorBound} for TLBT & $\boldsymbol{\sigma}_r$ TLBT& $\boldsymbol{\sigma}_r$ BT \\
		\hline
		9.18e-04& 4.37e-04& 1.92e-07& 7.04e-07& 3.1e-03 \\
		\hline
	\end{tabularx}\label{tab:SmallExample_bounds}
\end{table} 
 
 \begin{table}[t]
 	% \footnotesize
 	\centering
 	\caption{Constants for asymptotic behavior}\footnotesize
 	\begin{tabularx}{0.7\textwidth}{|X|l|l|l|l|l|}
 		\hline
 		& $c$& $\lambda$ & $\hat{c}$& $\hat{\lambda}$& Thm 3.2 bound\\
 		\hline
 		Eig. Value Decomp. & 12.26 & 0.97& 2.95&  0.97&  276.91\\
 		\hline
 		Field of values & 2.41& 1.06& 2.41& 0.99&  9.93\\
 		\hline
 	\end{tabularx}\label{tab:SmallExample_Asymp}
 \end{table} 
 
%\red{\text{Should I write this ?}}In what follows, we propose a method to compute the Gramians and TLBT for large-scale systems.

%%%%%%%%%%%%%%%%%%%%%%%%%%%%%%%%%%%%%%%%%%%%%%%%%%%%%%
\section{Computational Aspects}\label{sec:compute}
 \subsection{Numerical Computation of the Gramians}
As for BT for continuous-time systems, the solution of the large-scale discrete-time Lyapunov equations \eqref{eq:InfSteinEq}, \eqref{eq:FiniteSteinEq} is the
computationally most demanding step. We will restrict the following discussion to the controllability Gramians, since from there the results for the
observability are easily given by replacing $A,~B$ by $A^T,~C^T$.
Especially
in the large-scale situation, directly computing and storing the Gramians is infeasible because, in general, they are large, dense matrices. 
The common practice when $m,p\ll n$ is to compute approximations of low-rank, e.g. $P_{\infty}\approx QYQ^T$ with
$Q\in\Rnk$, $Y=Y^T\in\Rkk$, $k\ll n$ which is motivated by the typically rapid singular value decay of the Gramians, see e.g.,
% ~\cite{morAntSZ02,Sab07,
~\cite{BenKS12d,Sad12,BecT17}. BT is then carried out with the low-rank solution factors of the (infinite or time-limited) Gramians instead of exact Cholesky factors.

There exist
different algorithms for computing the low-rank solution factors $Q,~Y$ by using techniques from large-scale, numerical linear
algebra. For the Stein equations, the expressions~\eqref{eq:InfGramians} directly motivates the Smith method~\cite{Smi68,Pen00} for computing low-rank factors:
\begin{align}\label{smith}
\begin{split}
 P_{k}&=AP_{k-1}A^T+BB^T,\quad k\geq 1,~X_0=0\\
 &=\sum\limits_{j=0}^{k-1} A^jBB^T(A^T)^j=Z_{k-2}Z_{k-2}^T+A^{k-1}BB^T(A^T)^{k-1}=Z_kZ_k^T\approx P_{\infty},
 \\Z_k&:=[B,AB,\ldots,A^{k-1}B].
 \end{split}
\end{align}
Underlying the Smith iteration~\eqref{smith} is the (block) Krylov subspace of order $k$:
 \begin{align*}
   \range{Z_k}=\cK_{k}(A,B)=\range{[B,AB,\ldots,A^{\tau}B]}.
 \end{align*}
Hence, we can also find approximate solutions of~\eqref{eq:InfGramians} via a block Arnoldi process~\cite{Saa90,morJaiK94}.  Let
$Q_k=[q_1,\ldots,q_k]\in\R^{n\times km}$, $q_i\in\R^{n\times m}$ span a orthonormal basis of
$\cK_{k}(A,B)$ with $B=q_1\beta$, $\beta\in\R^{m\times m}$. Suppose the Arnoldi relation $AQ_k=Q_kH_k+q_{k+1}h_{k+1,k}E_k^T$ holds, where
$H_k=Q^T_kAQ_k=[h_{ij}]$ is block upper Hessenberg, and $E_k=e_k\otimes I_m$. Then $Z_k=Q_kR_k$ for a block upper triangular matrix
$R_k=[r_1,\ldots,r_k]\in\R^{mk\times mk}$ and
$r_i=H_k(r_{i-1})$, $2\leq i\leq k$, $r_1=E_1\beta$. Algorithm~\ref{alg:smith} illustrates this procedure.
\begin{algorithm}[t]
  \SetEndCharOfAlgoLine{} \SetKwInOut{Input}{Input}\SetKwInOut{Output}{Output}
  \caption[Smith-Arnoldi method]{Smith-Arnoldi method for Stein equations}
  \label{alg:smith}
  % \begin{algorithmic}[1]
  \Input{$A,~B$ as in~\eqref{eq:InfGramians},
    tolerances $0<\varepsilon\ll1$.}
  \Output{$Q_kQ_k^T\approx P_{\infty}$ with $Q_k\in\R^{n\times \ell}$, $Y_k\in\R^{\ell \times \ell}$, 
    $\ell\leq mk\ll n$.}
  $B=q_1\beta$ s.t. $q_1^Tq_1=I_m$, $Q_1=q_1$, $r_{1:m,1:m}=\beta E_1$\;\nllabel{arn_ini} 
  \For{$k=1,2,\ldots$} {%
   $g=Aq_k$, $h_{k+1,1:k}=Q_k^Tg$, $g_+=g-Q_kh_{k+1,1:k}$.\;  
   $q_{k+1}=g_+h_{k+1,k}$ s.t. $q_{k+1}^Tq_{k+1}=I_m$.\nllabel{Arn_ort}\;
    $Q_{k+1}=[Q_k,~q_{k+1}]$.\; 
    $r_{1:k+1,1:k+1}=H_{1:k+1,1:k}r_{1:k+1,1:k}$ (next block column of $R_k$)\;
    }
    $Q_k=Q_kR_k$.
\end{algorithm}
Alternatively, we can impose a Galerkin condition on the Lyapunov residual for $Q_kY_kQ_k^T$ enforcing that $Y_k$ is the
solution of a projected version of~\eqref{eq:InfGramians}, i.e., 
\begin{align}\label{reducedStein}
H_kY_kH_k^T-Y_k+(Q_k^TB)(B^TQ_k) =0, \quad H_k:=Q_k^TAQ_k,
\end{align}
 which can be solved by standard
dense methods. If the quality of the approximation $Q_kY_kQ_k^T$ is not sufficient, $Q_k$ is orthogonally expanded by continuing the Arnoldi process. 
This convergence rate of the Smith iteration depends on the spectral radius of $A$ and can be very slow if $\rho(A)\approx 1$. To overcome this issues, so
called squared Smith methods were discussed in~\cite{Pen00,BenKS12d,LiWCetal12} with limited success.

The occurrence of the matrix functions in time- and frequency limited BT, or more precisely the action of $f(A)$ to $B$, adds an additional computational
difficulty. 
At a first glance, the required monomials $f(z)=z^{\tau}$ in time-limited discrete-time BT appear to be a
comparatively simple situation, especially if $\tau$ is very small relatively to $n$ and the required $\tau$ matrix vector products
with $A$ are affordable. In that case we can directly use the iteration~\eqref{smith} or Algorithm~\ref{alg:smith} for 
the time-limited Gramians~\eqref{eq:TLGramians}. Running~\eqref{alg:smith} for an additional step allows to read off $F$ from the last block column of $Q_k$.
Alternatively, we can use the Galerkin projection framework mentioned above, i.e.,  we build Galerkin approximations 
$F\approx H_k^{\tau}(Q_k^TB)$ and $P_{\tau}\approx Q_kY_kQ_k^T$, where $Y_k$ solves
\begin{align}\label{reducedTLGram}
H_kY_kH_k^T-Y_k+(Q_k^TB)(B^TQ_k)- H_k^{\tau}(Q_k^TB)(B^TQ_k)(H_k^{\tau})^T=0.
\end{align}
These approximations are exact if $\range{Q_k}=\cK_{\tau+1}(A,B)$ because then $\range{A^{\tau}B}\in\range{Q_k}$ and $\range{Z_{\tau+1}}\in\range{Q_k}$ with
$Z_{\tau+1}$ from~\eqref{smith}.

Unfortunately, for large values $\tau\approx n$ and if $\rho(A)\approx 1$ this Galerkin approach or Algorithm~\ref{alg:smith} become impractical as they would
require prohibitively large subspace dimensions. Note that getting the powers of $A$ via approaches like binary powering~\cite[Chapter 4.1]{Hig08} are not
feasibly for large $A$, since successively squaring $A$ destroys its sparsity and, hence, the matrix-matrix multiplications become too costly.

 For achieving a faster convergence rate, i.e. accurate approximations with smaller subspace dimensions, rational Krylov subspaces
\begin{align}\label{ratkrylov}
 \range{Q_k}=\cR\cK_{k}(A,B,\boldsymbol{\xi})=\range{[B,(A+\xi_2I)^{-1}B,\ldots,\prod\limits_{j=2}^{k}(A+\xi_jI)^{-1}B]}
\end{align}
 have been proven to be a viable choice~\cite{DruS11,DruKS11,Bec11}, provided adequate shift parameters $\xi_i\in\C$ are available.
The majority of literature regarding rational Krylov methods for solving
large matrix equations is focused on the continuous-time case and, although the discrete-time case can be dealt with similarly, to the authors
knowledge not much is known about the shift selection. 
A low-rank ADI iteration for Stein equations~\eqref{eq:InfGramians} was proposed in~\cite{BenK14} and later improved in \cite{Kue16}. It is related to both the
Smith method as well as to rational Krylov subspaces. Both rational Krylov und ADI methods for~\eqref{eq:InfGramians} can be used directly
to~\eqref{eq:FiniteSteinEq} if $F$ and $G$ or approximations thereof are given
which, however, is a crucial point because they have to be computed first.

In the present work, we follow an approach similar to the one proposed
in~\cite{morBenKS16,Kue16,morKue18} for the continuous-time setting.  We generate a basis of the rational Krylov subspace~\eqref{ratkrylov} and  
solve the compressed Stein equation~\eqref{reducedTLGram} for $Y_k$ to acquire a low-rank approximation $P_{\tau,k}= Q_kY_kQ_k^T\approx
P_{\tau}$. The basis is expanded (the rational Arnoldi process is
continued) until $P_{\tau,k}$ is of the desired accuracy. As we will see, this allows to jointly approximate $F$ and $P_{\tau}$. 
The rational Krylov method for~\eqref{eq:TLGramians} is shown in
Algorithm~\ref{alg:rksm_tlimdale} and we shall next describe some important aspects this method. Obviously, by omitting all parts related to $F=A^{\tau}B$
Algorithm~\ref{alg:rksm_tlimdale} is applicable to the infinite Gramians~\eqref{eq:InfGramians} as well.

\begin{algorithm}[t]
  \SetEndCharOfAlgoLine{} \SetKwInOut{Input}{Input}\SetKwInOut{Output}{Output}
  \caption[Rational Krylov subspace method for time-limited DALEs]{Rational Krylov subspace
    method for time-limited DALEs~\eqref{eq:FiniteSteinEq}}
  \label{alg:rksm_tlimdale}
  % \begin{algorithmic}[1]
  \Input{$A,~B,~\tau$ as in~\eqref{eq:FiniteSteinEq},
    tolerances $0<\varepsilon\ll1$.}
  \Output{$Q_kY_kQ_k^T\approx P_{\tau}$ with $Q_k\in\R^{n\times \ell}$, $Y_k\in\R^{\ell \times \ell}$, 
    $\ell\leq mk\ll n$.}
  $B=q_1\beta$ s.t. $q_1^Tq_1=I_m$, $Q_1=q_1$.\;\nllabel{rksm_ini} 
  \For{$k=1,2,\ldots$} {%
    $H_k=Q_k^TAQ_k$, $B_k=Q_k^TB$.\;  
    $\hF_{k}=H_k^{\tau}B_k$, $F_k=Q_k\hF_{k}$\nllabel{rksm_Fk}.\;
    Compute Gramian defined by $H_k,~B_k$, $\hF_{k}$ (e.g., solve \eqref{reducedTLGram}).\nllabel{rksm_small}\;
%     \If{$\|\hF_{k}\hF_{k}^T-\hF_{k-1}\hF_{k-1}^T\|/\|\hF_{k-1}\|^2<\varepsilon_f$\nllabel{rksm_Fk_test}}{%
%     Solve
%       $H_kY_kH_k^T-Y_k+B_kB_k^T-\hF_{k}\hF_{k}^T=0$ for $Y_{k}$.\nllabel{rksm_Pj}\;
      Set $\fR_k:=\|A(Q_kY_kQ_k^T)A^T-(Q_kY_kQ_k^T)+BB^T-F_{k}F_{k}^T\|$ with $F_k\approx A^{\tau}B$.\;
      \If{$\fR_k<\varepsilon\|BB^T-F_{k}F_{k}^T\|$\nllabel{rksm_Lj_test}} {%
      Return $P_{\tau,k}=Q_kY_kQ_k^T$ (truncate if necessary).\;}
    Select next shift $s_{k+1}$.\;    
    Solve $(A-s_{k+1}I)g=q_k$ for $g$.\nllabel{rksm_linsys}\;
%     Real, orthogonal expansion of $Q_k$: 
    $g_+=g-Q_k(Q_k^Tg)$, $q_{k+1}=g_+\beta_k$ s.t. $q_{k+1}^Tq_{k+1}=I_m$.\nllabel{rksm_orth}\;
    $Q_{k+1}=[Q_k,~q_{k+1}]$.\;  
    }
\end{algorithm}
 
 \paragraph{Solution of the projected problems}
 Two approaches for dealing with~\eqref{reducedTLGram} in line~\ref{rksm_small} are discussed. Following the algorithmic strategy proposed
in~\cite{morBenKS16,Kue16,morKue18}, at first an approximation of $F=A^{\tau}B$ is computed by a projection principle: $F\approx F_k:=Q_k\hF_{k}$,
$\hF_{k}:=H_k^{\tau}(Q_k^TB)$. Since $H_k$ is of size $mk\ll n$, the powers of $H_k$ can be efficiently computed by binary powering which requires $\lfloor
\log_2{\tau}\rfloor$ matrix-matrix multiplications. This can be less costly compared to the computation of the more complicated matrix
functions (matrix exponentials and logarithms) that occur in the continuous-time case.
Since the goal is to approximate the associated
term of the inhomogeneity of~\eqref{eq:TLGramians}, we use a relative norm wise change $\fF:=\|\hF_{k}\hF_{k}^T-\hF_{k-1}\hF_{k-1}^T\|/\|\hF_{k-1}\|^2$ to
assess the accuracy of the current approximation $F_k$. 
Once $\fF\leq \varepsilon_f\leq \varepsilon\ll 1$, we consider this approximation as of sufficient accuracy and start solving the Galerkin
system~\eqref{reducedTLGram} for $Y_k$. This can be done by, e.g., by direct (Bartels-Stewart type) methods~\cite{Bar77} or the Smith iteration~\ref{smith}. The
rational Krylov method is continued until the scaled residual norm $\fR$ with respect to the low-rank solution $Q_kY_kQ_k^T$ falls below a given threshold
$\varepsilon$. In all these further steps, the quality of the approximation $F_k$ of $F$ can be further refined by computing a new $\hF_{k}$ before
solving the compressed Stein equation~\eqref{reducedTLGram}.

Alternatively, the separate computation of $\hF_{k}$ can be avoided since~\eqref{reducedTLGram} can be entirely dealt with by $\tau$ steps of the Smith
iteration~\ref{smith}. Depending on the sizes of $\tau$ and $H_k$, this can be less costly than the first approach and solving~\eqref{reducedTLGram} by a
direct method. By the discussing before, $\hF_k=H_k^{\tau}B_k$ can still be obtained as byproduct of the Smith iteration. Note that for the infinite situation
$\tau=\infty$, using the Smith iteration for~\eqref{reducedStein} requires that the restriction $H_k$ is stable which is theoretically ensured if the
numerical range of $A$ is in the unit disc.

The computational effort resulting from either of these two strategies can be further reduced by solving~\eqref{reducedTLGram} only in each $\mu$th step
(e.g., $\mu=5$) of Algorithm~\ref{alg:rksm_tlimdale}.

\paragraph{Computing the residual of the Stein equations}
Directly computing $\fR$ for assessing the accuracy of the low-rank approximation $Q_kY_kQ_k^T$ is impractical since the Lyapunov residual matrix is a large,
dense matrix. The following Lemma reveals an efficient way to compute the residual norm.
\begin{mylemma}
The residual matrix at step $k$ of Algorithm~\ref{alg:rksm_tlimdale} is given by
 \begin{align*}
 R_k&:=A(Q_kY_kQ_k^T)-Q_kY_kQ_k^T+BB^T-Q_k\hF_k\hF_k\hQ_k^T\\
 &=[g_k,w_k]\smb\psi^T_kY_k\psi_k&I_m\\I_m&0_m\sme[g_k,w_k]^T,\quad g_k:=s_{k+1}q_{k+1}-(I-Q_kQ_k^T)Av_{k+1},\\
 w_k&:=Q_kH_kY_k\psi_k,\quad \psi_k=\Psi_k^{-T}E_k\psi_{k+1,k},
\end{align*}
where $\Psi_k=[\psi_{ij}]\in\R^{km\times km}$ is the matrix of orthonormalization coefficients $\psi_{ij}\in\R^{m\times m}$ accumulated from
line~\ref{rksm_orth}. Hence, $\|R_k\|=\|S_k\smb\psi^T_kY_k\psi_k&I_m\\I_m&0_m\sme S_k^T\|$ and $S_k\in\R^{2m\times 2m}$ is the triangular factor of a thin
QR-factorization of $[g_k,w_k]$. The result is also valid for the infinite Gramians by omitting the term $Q_k\hF_k\hF_k\hQ_k^T$.
\end{mylemma}
\begin{proof}
 The result can be easily established by following the derivations of the associated results for the rational Arnoldi for continuous-time
equations~\cite{DruS11} and combining those with results regarding the standard Arnoldi methods for discrete-time equations~\cite{morJaiK94,BouHJ11}. 
\end{proof}

\paragraph{Shift parameter selection}
 Having suitable shift  parameters $\xi_2,\ldots,\xi_{k}$ available is crucial for a rapid convergence of the rational Arnoldi method.
 Two selection strategies are employed here. At first, alternating shifts $\xi_{j}=(-1)^{j}$, $2\leq j\leq k$ are
used\footnote{Communicated by Stefano Massei (EPF Lausanne).} which formally corresponds to the extended Krylov subspace setting ($\xi_{2j-1}=\infty$, $\xi_{2j}=0$) from the continuous-time case~\cite{Sim07},
and applying a Cayley transformation to map $\overline{\C_-}$ into the closed unit disc. Since with this choice, only two different coefficient matrices $A\pm
I$ occur in the linear systems in line~\ref{rksm_linsys}, precomputing and afterwards reusing sparse LU-factorizations $L_{\pm}U_{\pm}=A\pm
I$ in each step can substantially reduce the computation times for solving the linear systems.

\smallskip
The second shift selection approach is more general and
modifies the strategy proposed in~\cite{DruS11} by selecting shifts adaptively from
the boundary of the unit disc. Suppose the unit circle is discretized into $h_s\in\N$ points, 
$\Xi:=\lbrace \exp{\frac{\jmath 2\Pi i}{h_s}},~1\leq i\leq h_s\rbrace$, set $m=1$ for simplicity, and let $\theta_i\in\Lambda(H_k)$. 
The next shift $s_{k+1}$ is then obtained by maximizing the rational function associated to the current space $\cR\cK_{k}$, i.e.,
\begin{align}\label{rk_adapt}
 s_{k+1}=\argmax_{\xi\in\Xi} |r_k(\xi)|,\quad r_k(\xi)=\prod\limits_{j=1}^k\tfrac{\xi-\theta_j}{\xi-\xi_j},
\end{align}
For $m>1$, there are $mk$ Ritz values $\theta_i$ and each previous shift $s_j$ is taken $m$ times in the denominator of $r_k$ in~\eqref{rk_adapt} as
in~\cite{DruS11}. In the second variant the returned shifts are complex numbers. By demanding that each complex shift is followed by its complex conjugate, the
amount of complex arithmetic operations can be reduced following the machinery in, e.g.,~\cite{Ruh94b}, that will ensure the construction of real, low-rank
solution factors $Q_k$, $Y_k$.
\paragraph{Generalized systems}
The generalized Stein equations~\eqref{eq:genStein} corresponding to~\eqref{eq:gLTI} are handled as in the continuous-time setting by implicitly using the algorithms on an equivalent standard state-space system defined by, e.g., $A_M:=L_M^{-1}AU_M^{-1}$, $B_M:=L_M^{-1}B$, $C:=CU_M^{-1}$ with a precomputed sparse factorization $M=L_MU_M$. Afterwards, the obtained low-rank solution factors $Z$ have to be transformed back via $U_M^{-1}Q_k$. 
%%%%%%%%%%%%%%%%%%%%%%%%%%%%%%%%%%
 \subsection{Computing the error bounds}
We will briefly discuss the practical usage of the mentioned error bounds.
For computing the error bounds in~\eqref{eq:InfHorErrorBound} and Proposition~\ref{prop:FirstErrorBound} after a reduction of large-scale systems, the low-rank Gramian approximation are used in the associated places, e.g., in $\trace{B^TQ_{\tau}B}\approx\trace{B^TZ_{Q_{\tau}}Z_{Q_{\tau}}^TB}$. The computation of the terms involving the Gramians of the reduced order model requires solving Stein equations of dimension $r$ which can be done direct, dense methods. The terms involving the mixed Gramians require solving Sylvester equations~\eqref{eq:SylvStein_inf}, \eqref{eq:SylvesterSteinequations} where one defining coefficient is large and sparse but the other small and dense. There are specialized solvers for this particular situation, e.g.~\cite{morBenKS11}, that require $r$ sparse linear systems to be solved. One should be aware of that, since only approximate Gramians are used, the expression involving the controllability Gramians might not be identical to the expression with the observability Gramians. For the error bound in Proposition~\ref{prop:FirstErrorBound} for time-limited BT, this effect might be more pronounced since also only approximations of $F,G$ are available in practise which enter~\eqref{eq:SylvesterSteinequations}. In the upcoming we, therefore, use the average of both expressions. Another frequent observation is that the traces with positive and negative signs are very close to each other, e.g.,  $\trace{CP_{\tau}C^T +C_1\hP_{\tau}C_1^T} \approx \trace{2CYC_1^T }$, which can lead to numerical cancellation or even negative values for the complete trace. This seems to be especially an issue if the reduced order model is already very accurate. Hence, we take absolute values $\vert\trace{CP_{\tau}C^T +C_1\hP_{\tau}C_1^T-2CYC_1^T }\vert$ to circumvent these effects.

It is clear that the bound in Theorem~\ref{theo:BTh2TLbounds} is not accessible for large-scale system because the neglected quantities such as $B_{2}$,  $A_{12}$, $C_{2}$ are not available in a practical implementation of TLBT.

Some of these unknown quantities are also present in the bound in  Theorem~\ref{theo:SepBounds}. Additionally, the constants $c,\hat c$, $\lambda$, $\hat \lambda$ are required. Here, the approach used for bounding the powers of $A$, $A_{11}$ matters. When 
$\lambda$, $\hat \lambda$ represent the largest magnitude eigenvalues they can be easily computed for $A_{11}$ and estimated for $A$ by, e.g., an Arnoldi process. The constants $c,\hat c$ are then the condition numbers of the eigenvector matrices, which is a difficult to get quantity for large matrices unless the matrices are normal, i.e. $c=1$. Applying the Crouzeix-Palencia result~\cite{crouzeix2017numerical}, however, simply uses $c=\hat c=1+\sqrt{2}$ and the largest value of $z^{\tau}$ on the numerical range of $A$. It holds $\sup\vert z^{\tau}\vert\leq \alpha^{\tau}$, where $\alpha:=\sup\vert z\vert$ is the numerical radius of $A$ which can also be efficiently estimated by approaches utilizing an Arnoldi process see, e.g.,~\cite{Hoc11}.

%%%%%%%%%%%%%%%%%%%%%%%%%%%%%%%%%%
\section{Numerical Experiments}\label{sec:NumExp}
In this section we test the model order reduction methods and the algorithms for acquiring low-rank factors of the Gramians. 
All experiments are carried out with implementations in \matlab{}~2016a on a \intel\coretwo~i7-7500U CPU @ 2.7GHz with 16 GB RAM.
\subsection{Used test cases}
Since the majority of model reduction literature discussed the continuous-time situation, there is only a comparatively limited supply of test cases available.
We use some discrete-time systems from \cite{GugGO13} as well as artificially generated and freely scalable test systems, summarized with some additional info
in Table~\ref{tab:ex} including the spectral radius $\rho=\max\limits_{z\in\Lambda(A,M)}\vert z\vert$.
Consider a positive definite, diagonally dominant matrix $S=L+U+D$, where $L,U$ and $D$ are its strictly upper, lower, and, respectively, diagonal part. Here, $S$ is the matrix associated to a centered finite difference discretization of the Laplace
operator on the unit disc. The Jacobi (\textit{Jac})  iteration $v_{k+1}=D^{-1}(L+U)v_k+D^{-1}b$ for the linear system $Sv=b$ represents a basic generalized discrete-time
system with coefficients $A=L+U$ and $M=D$ satisfying $\Lambda(D^{-1}(L+U))=\Lambda(L+U,D)\subset\D$, see, e.g.~\cite[Chapter 11.2]{GolV13}. Likewise, the Gauss-Seidel (\textit{GS}) iteration is given by $A=L$, $M=U+D$ with $\Lambda(L,U+D)\subset\D$ but not requiring diagonal dominance of $S$. The input and output maps $B,C$ for the  \textit{Jac}, \textit{GS} examples are chosen randomly from a uniform distribution on $[0,1]$.
\begin{table}[t]
  % \footnotesize
  \centering
  \caption{Overview of examples}\footnotesize
  \begin{tabularx}{\textwidth}{l|l|l|X|l}
%     \hline
    Example & $n$ & $m,~p$&details&$\rho$\\
    \hline%\hline
    \textit{skl}&24389&4, 6&discrete-time system "sparse-skewlap3d-mod-1" from~\cite{GugGO13}\footnote&0.91372\\
     %\textit{sprand}&10000&4, 6&discrete-time system "sparse-random" from~\cite{GugGO13}\footnote[5]&&&1e-8\\
    \textit{Jac}&31064&5, 5&Jacobi iteration for $S=$\texttt{delsq(numgrid('D',200))}&0.99985\\
    \textit{GS}&31064&5, 5&Gauss-Seidel iteration for $S$&0.9997\\
%     \textit{cd2d\_ie}&20082&5&implicit euler model of &&&1e-10\\
  \end{tabularx}\label{tab:ex}
\end{table}
\footnotetext[5]{Available at~Mert
G{\"u}rb{\"u}zbalaban's webpage \url{http://mert-g.org/software/}}

\subsection{Approximation of Gramians and matrix powers}
We start testing the approximation of the infinite and time-limited Gramians as well as $F=A^{\tau}B$ by the methods described in Section~\ref{sec:compute}: the Smith method from Algorithm~\ref{alg:smith} and the rational Krylov subspace method (Algorithm~\ref{alg:rksm_tlimdale}) using two types of shifts: $\xi_{j}=(-1)^{j}$ (RKSM($\pm1$)) and the adaptive selection on the unit circle (RKSM($\D$)). We also compare with the LR-ADI iteration for discrete-time Lyapunov equations~\cite{BenK14,Kue16}. For the time-limited equations~\eqref{eq:TLGramians} this is done via a hybrid approach, where the approximation $F$ obtained from RKSM($\pm1$) is used to set up the inhomogeneity. The time-limited Gramians are considered with two different time limits to gain insight how $\tau$ influences the computations.
The desired accuracy is for all cases is 
\begin{align*}
\fR:=\|AP_kA^T-P_k+BB^T-F_kF_k^T\|/\|BB^T-F_kF_k^T\|\leq \tau_P:=10^{-8}
\end{align*}
and $\|BB^T-F_kF_k^T\|/\|F_kF_k^T\|\leq \tau_f=10^{-8}$ is used for the approximation of $F$ computed in Algorithm~\ref{alg:rksm_tlimdale}.
The exception is the Smith method for the time-limited Gramians, which is carried out for exactly $\tau$ steps, hence providing exact results (up to round-off). After termination, the computed Gramians approximations are truncated by means of an eigenvalue decomposition and keeping only those eigenpairs with $\sum(\lambda_i(P)) > 10^{-12}\lambda_{\max}(P)$. The results are summarized in Table~\ref{tab:lrf_compare}, where the approximation of $F$ obtained by the Smith method is used for the final residual norms $\fR$ regarding the time-limited Gramians. Apparently, for small final times $\tau$ the Smith method can be competitive
 in terms of the computation time especially for the \textit{skl} and \text{jac} examples. It requires, e.g., the least amount of time for $\tau=50$ and the \textit{skl} example among all tested methods. Due to the comparatively small spectral radius of $A$ in the \textit{skl} example, the Smith method delivers also competitive times for the infinite Gramians, but fails to deliver the required accuracy for the other two examples. For all examples, the LR-ADI iteration appears to require the smallest computation times for the infinite Gramians. 
Considering the dimensions of the built up subspaces, however, indicates that the Smith method generates substantially larger spaces compared to the other approaches. Also the obtained ranks after truncation seem to be somewhat higher than for the other methods.
 For approximating the time-limited Gramians, the RKSM approaches seem to be a viable choice with respect to both computational time and size of the subspaces, especially for larger values of $\tau$. The used shift generation strategy has a noticable influence: while for the \textit{skl} example using the shifts $\pm 1$ leads to less consumed time than the shifts from the unit circle ($\D$), is is the other way around for the \textit{jac} example, and for the \textit{GS} example both shift approaches lead to similar results. The obtained subspace dimensions generated with RKSM($\D$) are almost all cases smaller compared to  RKSM($\pm 1$). The smaller computation times of RKSM($\pm 1$) for the  \textit{skl} and \textit{GS} examples are a result of the reuse of LU-factorizations of $A\pm M$ for the linear systems as explained before. For the \textit{jac} example these savings in solving the linear systems were nullified by the substantially higher subspace dimensions which resulted in much higher times for solving the projected problems.

However, for most examples the substantial discrepancy between subspace dimension and rank after truncation indicates that further enhancements by selecting better shifts are possible. We plan to pursue this topic in future research. 
For the time-limited Gramians, the hybrid approach of RKSM and LR-ADI appears to yield similar results than the pure RKSM approach.

To conclude this first experimental phase, for small final times $\tau$ (and/or a small spectral radius of $A$), the Smith method can be a viable choice for generating the low-rank factors of the (time-limited) Gramians. For larger $\tau$ (and/or spectral radii close to one), the rational Krylov approach appears to be superior, even with the basic shift selection strategies employed here. If $\tau=\infty$, the LR-ADI iteration is often the fastest method.

\begin{table}[t]
  % \footnotesize
  \centering
  \caption{Column dimension $d$ of built up low-rank factors before truncation, rank rk after truncation, final residual norm $\fR$, and computation
    time $t_c$ (in seconds) of the approximation of
    $P$, $P_{\tau}$ by different methods.}\scriptsize
  \begin{tabularx}{1\linewidth}{|l|X|l|l|l|l|l|l|l|l|l|l|l|l|}
    \hline
    \multicolumn{2}{|l|}{}&\multicolumn{4}{c|}{$P_{\infty}$}&\multicolumn{4}{c|}{$P_{\tau=50}$
    } &\multicolumn{4}{c|}{$P_{\tau=150}$}\\
    \cline{3-14}
    Ex.&method&$d$&rk&$\fR$&$t_c$&$d$&rk&$\fR$&$t_c$&$d$&rk&$\fR$&$t_c$\\
    \hline \multirow{4}{*}{%
      %\begin{minipage}{\linewidth}
        \textit{skl}
      %\end{minipage}
    }
&Smith&680&149&9.4e-09&27.4&200&105&1.5e-15&1.7&600&145&1.4e-15&15.0\\ 
&RKSM$(\pm1)$&140&91&2.1e-10&21.4&240&64&7.3e-13&9.5&240&85&4.9e-13&9.5\\ 
&RKSM$(\D)$&128&91&3.0e-10&42.3&168&64&7.3e-13&41.4&184&85&4.9e-13&45.6\\ 
&ADI&64&64&5.9e-09&15.5&176&69&3.1e-09&36.6&152&75&4.3e-09&32.6\\ 
    \hline
		%\multicolumn{2}{|l|}{}&\multicolumn{4}{c|}{$P_{\infty}$}&\multicolumn{4}{c|}{$P_{\tau=50}$
    %} &\multicolumn{4}{c|}{$P_{\tau=150}$}\\
    %%\cline{3-14}
    %%Ex.&method&$d$&rk&$\fR$&$t_c$&$d$&rk&$\fR$&$t_c$&$d$&rk&$\fR$&$t_c$\\
    %\hline \multirow{4}{*}{%
      %%\begin{minipage}{\linewidth}
        %\textit{srnd}
      %%\end{minipage}
    %}
%&Smith&1200&74&5.7e-03&32.4&200&76&1.8e-15&0.7&600&76&1.2e-15&5.5\\ 
%&RKSM$(\pm1)$&60&31&8.6e-09&30.4&60&42&1.2e-08&25.3&60&42&1.2e-08&25.5\\ 
%&RKSM$(\D)$&64&31&1.2e-09&233.3&88&42&2.3e-11&287.6&88&42&2.3e-11&291.1\\ 
%&LR-ADI(+RKSM)&72&57&7.5e-11&204.1&128&39&1.1e-09&212.5&144&42&1.1e-09&250.8\\ 
    %\hline
		\multicolumn{2}{|l|}{}&\multicolumn{4}{c|}{$P_{\infty}$}&\multicolumn{4}{c|}{$P_{\tau=100}$
    } &\multicolumn{4}{c|}{$P_{\tau=200}$}\\
    %\cline{3-14}
    %Ex.&method&$d$&rk&$\fR$&$t_c$&$d$&rk&$\fR$&$t_c$&$d$&rk&$\fR$&$t_c$\\
    \hline \multirow{4}{*}{%
      %\begin{minipage}{\linewidth}
        \textit{Jac}
      %\end{minipage}
    }
&Smith&1500&246&7.0e-01&347.3&500&201&3.2e-13&53.2&1000&230&7.4e-13&173.4\\ 
&RKSM$(\pm1)$&700&156&1.9e-09&381.7&700&111&1.0e-10&172.5&700&121&2.3e-10&179.0\\ 
&RKSM$(\D)$&630&156&2.1e-09&342.8&360&111&1.8e-10&52.6&380&121&2.3e-10&69.6\\ 
&ADI&230&218&8.9e-09&13.0&570&117&5.5e-09&92.9&540&127&7.0e-09&79.5\\ 
    \hline
		\multicolumn{2}{|l|}{}&\multicolumn{4}{c|}{$P_{\infty}$}&\multicolumn{4}{c|}{$P_{\tau=150}$
    } &\multicolumn{4}{c|}{$P_{\tau=250}$}\\
    %\cline{3-14}
    %Ex.&method&$d$&rk&$\fR$&$t_c$&$d$&rk&$\fR$&$t_c$&$d$&rk&$\fR$&$t_c$\\
    \hline \multirow{4}{*}{%
      %\begin{minipage}{\linewidth}
        \textit{GS}
      %\end{minipage}
    }
&Smith&1500&209&6.0e-01&291.4&750&191&3.9e-13&60.6&1250&206&5.7e-13&166.4\\ 
&RKSM$(\pm1)$&325&105&4.3e-09&188.4&375&92&8.6e-10&24.2&325&97&8.6e-10&18.5\\ 
&RKSM$(\D)$&410&105&4.3e-09&59.8&280&92&8.6e-10&20.4&280&97&1.1e-09&20.6\\ 
&ADI&135&135&3.8e-09&5.0&310&86&9.9e-09&25.1&290&97&3.0e-09&20.6\\  
    \hline
  \end{tabularx}\label{tab:lrf_compare}
\end{table}

\subsection{Model reduction results and error bounds}
Now we carry out infinite and time-limited balanced truncation employing low-rank Gramian approximations generated from the experiments before. It is noteworthy that, apart from different computations times, the obtained reduction results were largely unaffected by the employed method for generating the low-rank factors, provided the accuracy threshold was achieved.
Table~\ref{tab:mor_results} lists the results obtained by reducing the systems to different order $r$: the largest output error in the considered time interval $\cE_{\max}:=\max\limits_{0\leq k\leq  \tau}\|y(k)-\hy(k)\|_2$, the error bounds \eqref{eq:InfHorErrorBound} and Proposition~\ref{prop:FirstErrorBound} for BT and, respectively, TLBT, twice the sum of the neglected (time-limited) Hankel singular values, $2\sum\limits_{i\geq r+1}\sigma_i$, and the spectral radius of $A_r$ to access stability.  
%\textbf{Should we also give the bound in Proposition~\ref{prop:FirstErrorBound} for infinite BT?}
For selected reduced orders, Figures~\ref{fig:morerror_skl}--\ref{fig:morerror_GS} illustrate the output errors $\|y(k)-\hy(k)\|_2$ against time $k$ as well as the error bounds and HSV sums.
From the gathered data it is apparent that, in the targeted time interval $[0,\tau]$, TLBT achieves always smaller reduction errors and also the error bound from Proposition~\ref{prop:FirstErrorBound} takes smaller values than the counterpart~\eqref{eq:InfHorErrorBound} for unrestricted BT. The doubled sum of neglected HSVs is smaller for TLBT. All this is visually visible in Figures~\ref{fig:morerror_skl}--\ref{fig:morerror_GS}. As expected after passing the time limit $\tau$, the accuracy of the TLBT models worsens to a point $\tilde k\geq \tau$ where BT is more accurate.
We also observe from Table~\ref{tab:mor_results} that approximately half of the  reduced order model generated by TLBT are unstable. We see this, e.g. in Figures~\ref{fig:morerror_jac}--\ref{fig:morerror_GS}, where the output error drastically decreases after passing the time limit $\tau$.  Comparing the largest output errors $\cE_{\max}$ and the sums of neglected HSVs for TLBT in Table~\ref{tab:mor_results} suggest that, although a bound of the form~\eqref{eq:sumhsvBT} is not given for TLBT, the HSV sum could be used for adaptively determining suitable the reduced order $r$ as it is done in unlimited BT.  To underline this point, we repeat the model reduction experiment but let (TL)BT determine to reduced orders $r$ adaptively such that 
\begin{align}\label{adaptorderhsv}
2\sum_{k=r+1} \sigma_{k}\leq\epsilon_{hsv},
\end{align}
for different given reduction tolerances $0<\epsilon_{hsv}<1$. The results are summarized in Table~\ref{tab:mor_results_adapt} and indicate that this adaptive reduced order determination works for TLBT as fine as for unlimited BT. Moreover, TLBT appears to yield smaller reduced order models of similar accuracy compared to BT. This is a similar observation as for continuous-time TLBT~\cite{morRedK18}.

\begin{table}[t]
  % \footnotesize
  \centering
  \caption{Results and error bounds for BT and TLBT model reduction to fixed orders $r$.}\scriptsize
  \begin{tabularx}{1\linewidth}{|X|l|l|l|l|l|l|l|l|l|}
    \hline
    &&\multicolumn{4}{c|}{BT}&\multicolumn{4}{c|}{TLBT}\\
    \cline{3-10}
    Ex.&r&$\cE_{\max}$&bound&$\boldsymbol{\sigma}_r$&$\rho$ &$\cE_{\max}$&bound&$\boldsymbol{\sigma}_r$&$\rho$\\
		\hline \multirow{3}{*}{%
        \textit{skl}, $\tau=50$
    }
		&20&3.5e-01&9.3e-01&1.1e+01&0.9728&2.2e-01&6.9e-01&6.5e+00&0.9584\\ 
&40&1.0e-02&3.3e-02&2.4e-01&0.9717&1.3e-03&3.6e-03&2.2e-02&0.9852\\ 
&60&6.2e-05&1.8e-02&2.7e-03&0.9650&4.6e-07&6.9e-04&1.1e-05&1.0008\\ 
\hline
\multirow{3}{*}{%
        \textit{jac}, $\tau=200$
    }
&40&8.1e-01&2.7e+00&5.9e+01&0.99985&1.8e-01&6.5e-01&6.8e+00&1.00019\\ 
&60&1.9e-01&1.7e+00&6.7e+00&0.99986&8.9e-03&6.1e-01&1.9e-01&1.00030\\ 
&80&2.1e-02&2.0e+00&6.8e-01&0.99985&1.5e-04&6.6e-01&6.0e-03&1.00000\\ 
\hline
\multirow{3}{*}{%
        \textit{GS}, $\tau=150$
    }
&40&1.3e-01&2.2e+00&2.7e+00&0.99971&1.4e-02&3.6e-01&3.2e-01&0.99964\\ 
&60&5.5e-03&2.3e+00&1.0e-01&0.99971&2.2e-04&5.0e-01&4.6e-03&0.99988\\ 
&80&1.3e-04&1.5e+00&3.8e-03&0.99971&3.9e-06&7.2e-01&5.4e-05&1.00592\\ 
		\hline
  \end{tabularx}\label{tab:mor_results}
\end{table}

\begin{figure}[t]
% This file was created by matlab2tikz v0.4.7 running on MATLAB 9.1.
% Copyright (c) 2008--2014, Nico Schlömer <nico.schloemer@gmail.com>
% All rights reserved.
% Minimal pgfplots version: 1.3
% 
% The latest updates can be retrieved from
%   http://www.mathworks.com/matlabcentral/fileexchange/22022-matlab2tikz
% where you can also make suggestions and rate matlab2tikz.
% 
\begin{tikzpicture}
\pgfplotstableread{errorplot_skl_50_40.dat}\loadedtable

\begin{axis}[%
cycle list name=error,
width=0.951\wex,
height=\hex, %3.565625in,
scale only axis,
%separate axis lines,
every outer x axis line/.append style={white!15!black},
every x tick label/.append style={font=\color{white!15!black}},
xmin=0,
xmax=250,
xlabel={k},
every outer y axis line/.append style={white!15!black},
every y tick label/.append style={font=\color{white!15!black}},
ymode=log,
ymin=1e-08,
ymax=100,
yminorticks=true,
xmajorgrids,
ymajorgrids,
ylabel={error (bound)},
legend style={draw=white!15!black,fill=white,legend cell align=left,at={(0.98,0.03)},font=\small, anchor=south east},
legend entries={error BT, error TLBT,bound BT-Eq.~\eqref{eq:InfHorErrorBound},$\boldsymbol{\sigma}_r$ BT,bound TLBT-Prop.~\ref{prop:FirstErrorBound},$\boldsymbol{\sigma}_r$ TLBT}
]
\addplot table[x index=0,y index=4] {\loadedtable};

\addplot table[x index=0,y index=5] {\loadedtable};

\addplot table[x index=0,y index=6] {\loadedtable};

\addplot table[x index=0,y index=7] {\loadedtable};

\addplot table[x index=0,y index=8] {\loadedtable};

\addplot table[x index=0,y index=9] {\loadedtable};

\addplot [color=black,dashed]
  table[row sep=crcr]{50	1e-08\\
50	100\\
};
\end{axis}
\end{tikzpicture}%
\caption{Output errors $\|y(k)-\hy(k)\|_2$, error bounds, and sum of neglected HSVs $\boldsymbol{\sigma}_r$ for (TL)BT reduction of \textit{skl} example to order $r=40$ with time limit $\tau=50$.}\label{fig:morerror_skl}
\end{figure}
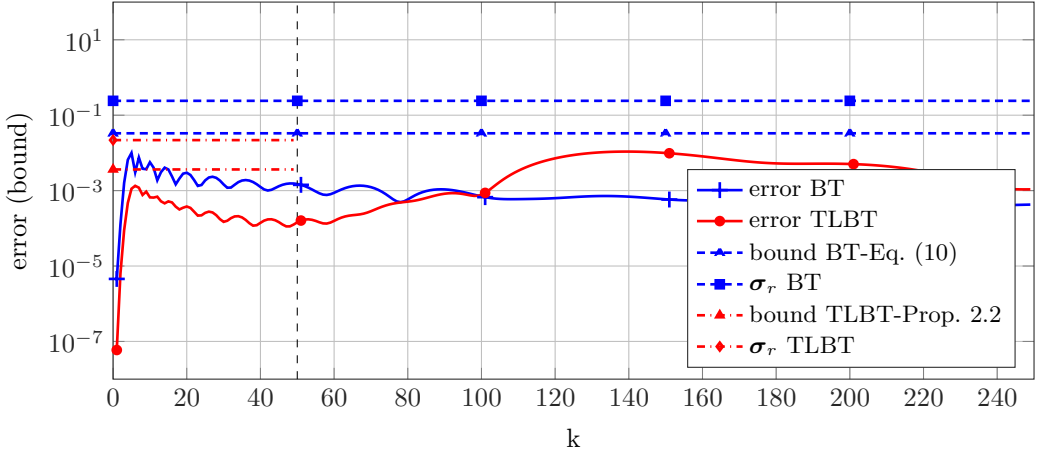
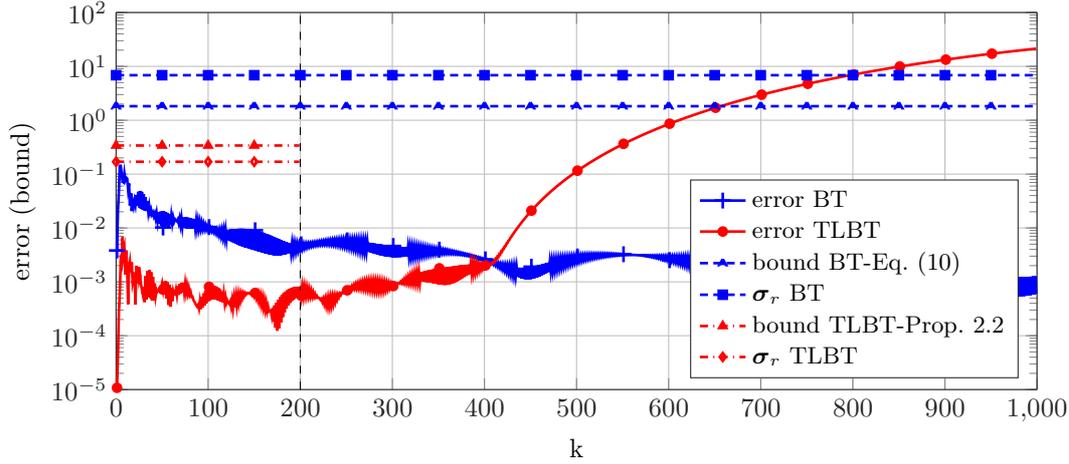
\begin{figure}[t]
% This file was created by matlab2tikz v0.4.7 running on MATLAB 9.1.
% Copyright (c) 2008--2014, Nico Schlömer <nico.schloemer@gmail.com>
% All rights reserved.
% Minimal pgfplots version: 1.3
% 
% The latest updates can be retrieved from
%   http://www.mathworks.com/matlabcentral/fileexchange/22022-matlab2tikz
% where you can also make suggestions and rate matlab2tikz.
% 
\begin{tikzpicture}
\pgfplotstableread{errorplot_jac_200_60.dat}\loadedtable

\begin{axis}[%
cycle list name=error,
width=0.951\wex,
height=\hex, %3.565625in,
scale only axis,
separate axis lines,
every outer x axis line/.append style={white!15!black},
every x tick label/.append style={font=\color{white!15!black}},
xmin=0,
xmax=1000,
xlabel={k},
every outer y axis line/.append style={white!15!black},
every y tick label/.append style={font=\color{white!15!black}},
ymode=log,
ymin=1e-05,
ymax=100,
yminorticks=true,
xmajorgrids,
ymajorgrids,
ylabel={error (bound)},
legend style={draw=white!15!black,fill=white,legend cell align=left,at={(0.98,0.03)},font=\small, anchor=south east,},
legend entries={error BT, error TLBT,bound BT-Eq.~\eqref{eq:InfHorErrorBound},$\boldsymbol{\sigma}_r$ BT,bound TLBT-Prop.~\ref{prop:FirstErrorBound},$\boldsymbol{\sigma}_r$ TLBT}
]
\addplot table[x index=0,y index=4] {\loadedtable};

\addplot table[x index=0,y index=5] {\loadedtable};

\addplot table[x index=0,y index=6] {\loadedtable};

\addplot table[x index=0,y index=7] {\loadedtable};

\addplot table[x index=0,y index=8] {\loadedtable};

\addplot table[x index=0,y index=9] {\loadedtable};

\addplot [color=black,dashed,forget plot]
  table[row sep=crcr]{200	1e-05\\
200	100\\
};
\end{axis}
\end{tikzpicture}%
\caption{Output errors $\|y(k)-\hy(k)\|_2$, error bounds, and sum of neglected HSVs $\boldsymbol{\sigma}_r$ for (TL)BT reduction of \textit{jac} example to order $r=60$ with time limit $\tau=200$.}\label{fig:morerror_jac}
\end{figure}
\begin{figure}[t]
% This file was created by matlab2tikz v0.4.7 running on MATLAB 9.1.
% Copyright (c) 2008--2014, Nico Schlömer <nico.schloemer@gmail.com>
% All rights reserved.
% Minimal pgfplots version: 1.3
% 
% The latest updates can be retrieved from
%   http://www.mathworks.com/matlabcentral/fileexchange/22022-matlab2tikz
% where you can also make suggestions and rate matlab2tikz.
% 
\begin{tikzpicture}
\pgfplotstableread{errorplot_GS_150_80.dat}\loadedtable

\begin{axis}[%
cycle list name=error,
width=0.951\wex,
height=\hex, %3.565625in,
scale only axis,
separate axis lines,
every outer x axis line/.append style={white!15!black},
every x tick label/.append style={font=\color{white!15!black}},
xmin=0,
xmax=700,
xlabel={k},
every outer y axis line/.append style={white!15!black},
every y tick label/.append style={font=\color{white!15!black}},
ymode=log,
ymin=1e-12,
ymax=100,
yminorticks=true,
xmajorgrids,
ymajorgrids,
ylabel={error (bound)},
legend style={draw=white!15!black,fill=white,legend cell align=left,at={(0.98,0.03)},font=\small, anchor=south east},
legend entries={error BT, error TLBT,bound BT-Eq.~\eqref{eq:InfHorErrorBound},$\boldsymbol{\sigma}_r$ BT,bound TLBT-Prop.~\ref{prop:FirstErrorBound},$\boldsymbol{\sigma}_r$ TLBT}
]
\addplot table[x index=0,y index=4] {\loadedtable};

\addplot table[x index=0,y index=5] {\loadedtable};

\addplot table[x index=0,y index=6] {\loadedtable};

\addplot table[x index=0,y index=7] {\loadedtable};

\addplot table[x index=0,y index=8] {\loadedtable};

\addplot table[x index=0,y index=9] {\loadedtable};

\addplot [color=black,dashed,forget plot]
  table[row sep=crcr]{150	1e-12\\
150	100\\
};
\end{axis}
\end{tikzpicture}%
\caption{Output errors $\|y(k)-\hy(k)\|_2$, error bounds, and sum of neglected HSVs $\boldsymbol{\sigma}_r$ for (TL)BT reduction of \textit{GS} example to order $r=80$ with time limit $\tau=150$.}\label{fig:morerror_GS}
\end{figure}
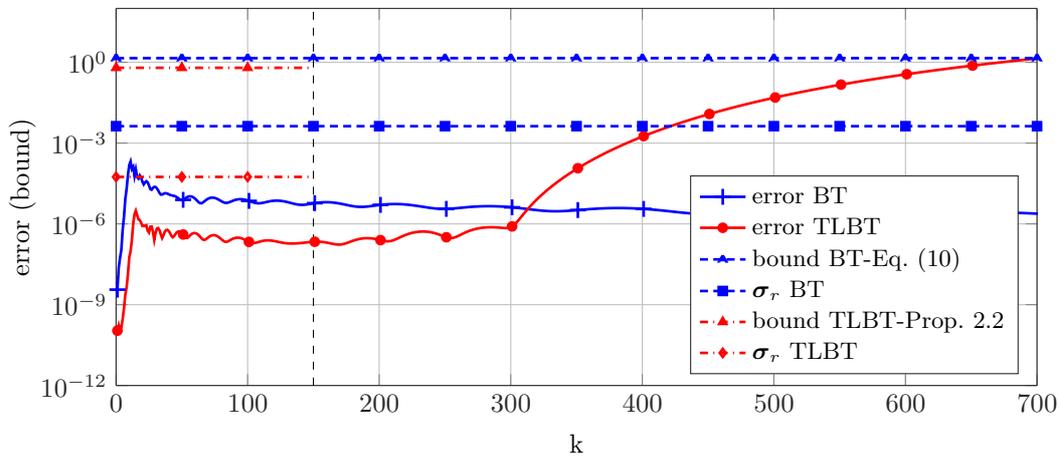

\begin{table}[t]
  % \footnotesize
  \centering
  \caption{Results and error bounds for BT and TLBT model reduction, where the reduced orders are adaptively determined via~\eqref{adaptorderhsv} for different $\epsilon_{hsv}$.}\scriptsize
  \begin{tabularx}{1\linewidth}{|X|l|l|l|l|l|l|l|l|l|l|l|}
    \hline
    &&\multicolumn{5}{c|}{BT}&\multicolumn{5}{c|}{TLBT}\\
    \cline{2-12}
    Ex.&$\epsilon_{hsv}$&$r$&$\cE_{\max}$&bound&$\boldsymbol{\sigma}_r$&$\rho$&$r$&$\cE_{\max}$&bound&$\boldsymbol{\sigma}_r$&$\rho$\\
		\hline
\multirow{3}{*}{%
        \textit{skl}, $\tau=50$
    }		
&1.0e-01&45&3.0e-03&2.0e-02&7.9e-02&0.96659&36&5.0e-03&8.8e-03&8.0e-02&0.96958\\ 
&1.0e-02&55&3.3e-04&1.8e-02&8.8e-03&0.96885&43&3.4e-04&3.6e-03&7.5e-03&0.97627\\ 
&1.0e-03&65&4.9e-05&1.8e-02&8.3e-04&0.96501&49&5.7e-05&1.8e-03&9.0e-04&0.97347\\ 		
\hline
\multirow{3}{*}{%
        \textit{jac}, $\tau=150$
    }		
&1.0e-01&97&2.2e-03&1.8e+00&9.5e-02&0.99985&64&2.1e-03&2.8e-01&9.2e-02&1.00513\\ 
&1.0e-02&115&1.6e-04&1.7e+00&9.3e-03&0.99985&78&2.2e-04&5.9e-01&8.9e-03&1.00049\\ 
&1.0e-03&133&2.7e-05&1.9e+00&9.4e-04&0.99985&90&2.6e-05&3.1e-01&9.4e-04&1.00059\\ 		
\hline
\multirow{3}{*}{%
        \textit{GS}, $\tau=150$
    }
&1.0e-01&61&5.5e-03&2.6e+00&9.9e-02&0.99971&46&7.1e-03&6.3e-01&9.8e-02&1.00073\\ 
&1.0e-02&75&2.1e-04&2.0e+00&9.3e-03&0.99971&58&3.3e-04&7.0e-01&8.4e-03&0.99987\\ 
&1.0e-03&88&2.5e-05&1.1e+00&9.7e-04&0.99971&69&2.4e-05&6.8e-01&8.1e-04&0.99981\\ 
		\hline
  \end{tabularx}\label{tab:mor_results_adapt}
\end{table}

\section{Conclusion}\label{sec:Conclusion}

In this paper, we study time-limited balanced truncation for discrete-time systems. The contributions o this work is dived in two parts. The first part is dedicated to developing output bounds for TLBT.  To this aim, we define the TL $h_2$ norm and its characterization using matrix equations. By means of this norm, we are able to express an error bound for the output. Afterward, we have analyzed the asymptotic behavior of these error bounds regarding the time-horizon, highlighting differences to the infinite time horizon as well as the continuous-time situation. The obtained bounds furthermore indicate that the neglected Hankel singular values can be used for an automatic reduced order determination.

The second part of this work is dedicated to computational aspects in  large-scale settings. Therein, the solution of  the TL Stein equations is obtained by using low-rank factorizations. Inspired by the continuous-time situation, rational Krylov subspace methods are proposed for computing the low rank solution factors. Furthermore, we discussed the residual and error bound computations as well as the selection of shift parameters for the rational Krylov subspace method. Finally, the algorithms are tested in large-scale examples, and the results are compared with other methods. The time-limited BT approach typically leads to more accurate ROMs in the considered time interval compared to infinite BT, which is also revealed by the smaller values of the corresponding output error bounds. As in the continuous-time case, TLBT occasionally returned unstable ROMS, which might be circumvented in investigations along the lines of, e.g.,~\cite{morGugA04}.
The proposed low-rank methods for the arising Stein equations returned satisfactory results for the application in the MOR context with respect to both computing time and accuracy. However, further research is required to bring them to the same level of efficiency as their continuous-time counterparts~\cite{DruS11,Kue16}. Especially the shift parameter selection for discrete-time problems should be improved in future research endeavours.

\section*{Acknowledgements}
Thank goes to Stefano Massei (EPF Lausanne) and Stefan Guettel (U Manchester) for helpful hints regarding the shift parameter selection for discrete-time problems, and to Michiel Hochstenbach (TU Eindhoven) for providing a \matlab{} routine for estimating the numerical radius of a matrix.
The larger part of this work was done while PK was still affiliated with the MPI Magdeburg.
\bibliographystyle{elsarticle-num} 
\bibliography{mor,igorBiblio,csc}

\end{document}